\documentclass[11pt,a4paper]{article}

\usepackage[margin=2.5cm]{geometry}
\usepackage{amsmath,amssymb,amsthm}
\usepackage{mathtools}
\usepackage{bm}
\usepackage{mathrsfs}
\usepackage{enumitem}
\usepackage{cite}
\usepackage{xcolor}
\usepackage[hypertexnames=false]{hyperref}
\hypersetup{colorlinks=true,linkcolor=blue!50!black,
            citecolor=blue!50!black,urlcolor=blue!50!black,
            pdftitle={Shape Holomorphy and Sparse Approximation of the
              Maxwell Electric Field Integral Operator},
            pdfauthor={Paul Escapil-Inchausp\'e and Carlos Jerez-Hanckes}}
\usepackage[capitalize,nameinlink]{cleveref}

\theoremstyle{plain}
\newtheorem{theorem}{Theorem}[section]
\newtheorem{corollary}{Corollary}[section]
\newtheorem{lemma}{Lemma}[section]
\newtheorem{proposition}{Proposition}[section]
\theoremstyle{definition}
\newtheorem{definition}{Definition}[section]
\newtheorem{assumption}{Assumption}[section]

\newtheorem{remark}{Remark}[section]
\numberwithin{equation}{section}
\crefname{assumption}{Assumption}{Assumptions}
\Crefname{assumption}{Assumption}{Assumptions}

\makeatletter
\let\orig@appendix\appendix
\def\appprefix@section{Appendix~}
\renewcommand{\appendix}{\orig@appendix
  \renewcommand{\@seccntformat}[1]{%
    \csname appprefix@##1\endcsname\csname the##1\endcsname.\quad}}
\makeatother

\newcommand{\refcite}[1]{\cite{#1}}

\newcommand{\Hpar}[2]{\boldsymbol H^{#1}_{\parallel}(#2)}

\providecommand{\R}{}\renewcommand{\R}{\mathbb R}
\providecommand{\C}{}\renewcommand{\C}{\mathbb C}
\providecommand{\N}{}\renewcommand{\N}{\mathbb N}
\providecommand{\U}{}\renewcommand{\U}{\mathbb U}
\providecommand{\Id}{}\renewcommand{\Id}{\operatorname{Id}}
\providecommand{\curl}{}\renewcommand{\curl}{\operatorname{curl}}
\providecommand{\divg}{}\renewcommand{\divg}{\operatorname{div}_{\Gamma}}
\providecommand{\divh}{}\renewcommand{\divh}{\operatorname{div}_{\widehat\Gamma}}
\providecommand{\divy}{}\renewcommand{\divy}{\operatorname{div}_{\Gamma_{\by}}}
\providecommand{\curlg}{}\renewcommand{\curlg}{\operatorname{curl}_{\Gamma}}
\providecommand{\cL}{}\renewcommand{\cL}{\mathcal L}
\providecommand{\by}{}\renewcommand{\by}{\mathbf y}
\providecommand{\bz}{}\renewcommand{\bz}{\mathbf z}
\providecommand{\bw}{}\renewcommand{\bw}{\mathbf w}
\providecommand{\bb}{}\renewcommand{\bb}{\mathbf b}
\providecommand{\bbeta}{}\renewcommand{\bbeta}{\bm\beta}
\providecommand{\brho}{}\renewcommand{\brho}{\bm\rho}
\providecommand{\bnu}{}\renewcommand{\bnu}{\bm\nu}
\providecommand{\bmu}{}\renewcommand{\bmu}{\bm\mu}
\providecommand{\bdelta}{}\renewcommand{\bdelta}{\bm\delta}
\providecommand{\rr}{}\renewcommand{\rr}{\mathbf r}
\providecommand{\ps}{}\renewcommand{\ps}{\bm\psi}
\providecommand{\jj}{}\renewcommand{\jj}{\mathbf j}
\providecommand{\vv}{}\renewcommand{\vv}{\mathbf v}
\providecommand{\uu}{}\renewcommand{\uu}{\mathbf u}
\providecommand{\EE}{}\renewcommand{\EE}{\mathbf E}

\providecommand{\pp}{}\renewcommand{\pp}{\mathbf p}
\providecommand{\ddir}{}\renewcommand{\ddir}{\mathbf d}
\providecommand{\obs}{}\renewcommand{\obs}{\bm\theta}
\providecommand{\nn}{}\renewcommand{\nn}{\mathbf n}
\providecommand{\xhat}{}\renewcommand{\xhat}{\widehat x}
\providecommand{\yhat}{}\renewcommand{\yhat}{\widehat y}
\providecommand{\Ghat}{}\renewcommand{\Ghat}{\widehat\Gamma}
\providecommand{\Xhat}{}\renewcommand{\Xhat}{\widehat X}
\providecommand{\Xhath}{}\renewcommand{\Xhath}{\widehat X_h}
\providecommand{\eps}{}\renewcommand{\eps}{\varepsilon}
\providecommand{\ds}{}\renewcommand{\ds}{\,\mathrm ds}

\providecommand{\norm}[1]{}\renewcommand{\norm}[1]{\left\|#1\right\|}
\providecommand{\abs}[1]{}\renewcommand{\abs}[1]{\left|#1\right|}
\providecommand{\pair}[2]{}\renewcommand{\pair}[2]{\left\langle #1,#2\right\rangle}
\providecommand{\pairt}[2]{}\renewcommand{\pairt}[2]{\left\langle #1,#2\right\rangle_{\tau}}

\providecommand{\supp}{}\renewcommand{\supp}{\operatorname{supp}}
\providecommand{\Et}{}\renewcommand{\Et}{\mathsf E}
\providecommand{\Vop}{}\renewcommand{\Vop}{\mathsf V}
\providecommand{\Top}{}\renewcommand{\Top}{\mathsf T}
\providecommand{\Aop}{}\renewcommand{\Aop}{\mathsf A}

\providecommand{\Iop}{}\renewcommand{\Iop}{\mathsf I}
\providecommand{\Pop}{}\renewcommand{\Pop}{\mathcal P}
\providecommand{\Rop}{}\renewcommand{\Rop}{\mathcal R}
\providecommand{\Hdiv}[1]{}\renewcommand{\Hdiv}[1]{\boldsymbol H^{-1/2}(\mathrm{div}_{#1},#1)}
\providecommand{\Hdivgen}[0]{}\renewcommand{\Hdivgen}[0]{\boldsymbol H^{-1/2}(\mathrm{div})}
\providecommand{\Hcurl}[1]{}\renewcommand{\Hcurl}[1]{\boldsymbol H^{-1/2}(\mathrm{curl}_{#1},#1)}

\begin{document}

\title{Shape Holomorphy and Sparse Approximation of the Maxwell\\
Electric Field Integral Operator}

\author{%
Paul Escapil-Inchausp\'e\\[2pt]
{\small Facultad de Ingenier\'ia y Ciencias, Universidad Adolfo Ib\'a\~nez, Santiago, Chile}\\
{\small\href{mailto:paul.escapil@edu.uai.cl}{\texttt{paul.escapil@edu.uai.cl}}}
\and
Carlos Jerez-Hanckes (corresponding author)\\[2pt]
{\small Department of Mathematics and Digital Futures, KTH Royal Institute of Technology,}\\
{\small Lindstedtsv\"agen 25, Stockholm, Sweden}\\
{\small Inria Chile, Santiago, Chile}\\
{\small\href{mailto:carlosjh@kth.se}{\texttt{carlosjh@kth.se}}}}

\date{}

\maketitle

\begin{abstract}
Time-harmonic Maxwell scattering by a perfectly conducting obstacle is governed by the electric field integral equation (EFIE), whose natural energy space is $\boldsymbol H^{-1/2}(\mathrm{div}_\Gamma,\Gamma)$. For countably parametrized surface deformations, both the operator and its energy space depend on the geometry. We establish operator-valued shape holomorphy of the EFIE operator on a fixed reference space and derive dimension-independent sparse approximation rates. The main analytical difficulty is a fractional mapping property absent from existing $L^2$-based operator-valued shape-holomorphy theory for weakly singular kernels: the Maxwell graph norm requires uniform holomorphy of the complex-deformed scalar single layer as an operator from $H^{-1/2}$ to $H^{1/2}$. We prove this one-order smoothing by realizing the Laplace principal part as the trace of a uniformly sectorial complex-coefficient divergence-form problem on a fixed real ambient domain. Combined with a surface contravariant Piola transformation, which identifies the geometry-dependent Maxwell trace spaces and cancels the surface Jacobians in the pulled-back EFIE exactly, this yields $(\mathbf b,p,\varepsilon)$-holomorphy of the EFIE operator family for $\ell^p$-summable shape deformations, $0<p<1$. Under pointwise exclusion of interior electric resonances, parameter-uniform invertibility is derived rather than assumed, and the surface current and far field inherit the same parametric regularity. Their Legendre coefficients are $\ell^p$-summable, yielding dimension-independent best $N$-term approximation rates. The operator-level results also transfer to Galerkin boundary element operators on a single reference mesh, with constants independent of the discretization dimension, and provide reusable surrogates for multiple incident fields and bounded linear observables.


\end{abstract}

\section{Introduction}\label{sec:intro}

\subsection{Problem and main contribution}

Boundary integral equations are a natural formulation of time-harmonic
Maxwell scattering in homogeneous media. For a perfectly electrically
conducting obstacle, the electric field integral equation (EFIE)
enforces the vanishing of the tangential total electric field on the
boundary and determines the induced surface current~\cite{Nedelec,BuffaHiptmair2003}.
It is posed in the tangential trace space
\(X(\Gamma)=\Hdiv{\Gamma}\), which controls a tangential distribution
together with its surface divergence. When the boundary depends on many
parameters, both the operator and its energy space therefore vary with
the geometry.

We consider affine-parametric deformations of a fixed reference surface
\(\Ghat\),
\begin{equation}\label{eq:intro-param}
\rr_{\by}(\xhat)=\xhat+\sum_{j\ge1}y_j\ps_j(\xhat),
\qquad \by=(y_j)_{j\ge1}\in\U=[-1,1]^{\N},
\qquad \Gamma_{\by}=\rr_{\by}(\Ghat),
\end{equation}
with amplitudes
\(b_j=\norm{\ps_j}_{W^{2,\infty}(\Ghat;\R^3)}\). We assume
\(\bb=(b_j)_{j\ge1}\in\ell^p(\N)\) for some \(0<p<1\). Such
anisotropic summability, combined with complex-parametric holomorphy,
is the standard mechanism behind dimension-robust sparse polynomial approximation~\cite{CohenDeVoreSchwab2011,ChkifaCohenSchwab2015,CohenDeVore2015}.

The difficulty is not merely to prove holomorphic dependence of a
Maxwell solution. Existing shape-holomorphy results for Maxwell
scattering are formulated at the field or solution level~\cite{JerezSchwabZech2017,AylwinJerezSchwabZech2020}, whereas the
operator-valued theory for weakly singular boundary kernels acts in \(L^2\)~\cite{DolzHenriquez2024}. The EFIE requires an operator-level
statement in \(\Hdiv{\Gamma}\). After the geometry-dependent trace
spaces are transported to \(\Ghat\), this reduces to a fractional
mapping property not supplied by the \(L^2\) theory: uniform
holomorphy of the complex-deformed scalar single layer as
\[
\Vop_{\bz}:H^{-1/2}(\Ghat)\longrightarrow H^{1/2}(\Ghat).
\]

Two ingredients remove this obstruction. The surface contravariant Piola
transformation identifies all spaces \(\Hdiv{\Gamma_{\by}}\) with the
single reference space \(\Xhat:=\Hdiv{\Ghat}\) and, at the same time,
cancels the surface Jacobians exactly in the EFIE variational form. The
uniform \(H^{-1/2}\to H^{1/2}\) mapping theorem for the complex-deformed
scalar single-layer family is then obtained by realizing its Laplace
principal part as the trace of a complex-coefficient Newton problem on a
fixed \emph{real} ambient space. There a physical elliptic problem is
available, even though the deformed surface is not.

Such operator-level regularity is what many-query Maxwell scattering
requires. In shape uncertainty quantification, geometric optimization, inverse
obstacle problems and surrogate modelling, the same boundary operator is
discretized for thousands of geometries (cf.~Refs.~\refcite{AylwinJerezSchwabZech2020,EscapilJerez2024,HiptmairScarabosioSchillingsSchwab2018,JerezSchwab2017}). A parametric surrogate of the {\em operator}, rather than of a solution for a prescribed excitation, can therefore be reused across incident fields and bounded linear observables.

\medskip
\noindent\textbf{Main results.} \Cref{thm:scalar} establishes uniform
\(H^{-1/2}(\Ghat)\to H^{1/2}(\Ghat)\) boundedness and operator-norm
holomorphy of the complex-deformed scalar single layer on admissible
tubes, by the fixed-domain trace--Newton construction of
\cref{sec:newton-route}. \Cref{prop:cancel} shows that the surface
Piola pullback cancels every surface Jacobian in the EFIE variational
form exactly. \Cref{thm:operator} combines these into
\((\bb,p,\eps)\)-holomorphy of
\(\by\mapsto\widehat\Aop_{\by}\in\cL(\Xhat,\Xhat')\).
\Cref{lem:uniforminverse} \emph{derives} parameter-uniform
invertibility from pointwise nonresonance, rather than assuming a uniform inverse bound as in Refs.~\refcite{JerezSchwabZech2017} and~\refcite{DolzHenriquez2024}.
\Cref{thm:current,thm:far} transfer the regularity to the surface
current and the far field, \cref{thm:sparse} gives
dimension-independent best \(N\)-term Legendre rates, and
\cref{prop:matrix,lem:disc-uniform} carry them to fixed-reference
Galerkin operators.

\subsection{Relation to previous work}\label{sec:intro-relation}

Two earlier lines of work on electromagnetic shape uncertainty must be
distinguished from the present one, since in each case the object of
study is different. In Ref.~\refcite{JerezSchwabZech2017}, the authors proved holomorphic dependence of the
scattered \emph{field} on bi-Lipschitz shape transformations, using
fixed-domain Maxwell variational problems. In the affine-parametric
setting their assumptions include conditions of the form
\begin{equation}\label{eq:JSZ-assumption}
\norm{\textstyle\sum_{j\ge1}\abs{DT_j}}_{L^\infty}<1,
\qquad
\big(\norm{T_j}_{W^{1,\infty}}\big)_{j\ge1}\in\ell^p,
\quad 0<p<1 .
\end{equation}
Subsequent work treats lossy cavities, multilevel algorithms,
quadrature and curved-domain approximation~\cite{AylwinJerezSchwabZech2020,AylwinJerezSchwabZech2023,AylwinJerez2021,AylwinJerez2023}.
Related acoustic results include
Refs.~\refcite{HiptmairScarabosioSchillingsSchwab2018}
and~\refcite{HiptmairSchwabSpence2025}.
Those results establish sparse approximability of a parameter-to-solution
map, for one incident field and one observable at a time. Here the
object of approximation is the boundary \emph{operator} itself, in its
energy norm. It is reusable across incident fields and observables, and
is the natural object for a parametric analysis of boundary element
discretizations.

The second line is perturbative. In Refs.~\refcite{JerezSchwab2017},~\refcite{EscapilJerez2020} and~\refcite{EscapilJerez2024}, shape differentiation of the boundary integral operators produces first-order
tensorized systems, whose accuracy is limited by truncation of the
deformation amplitude and whose rates are governed by a single shape
derivative. The present analysis removes that truncation: the full
countably parametric operator family is complexified, all mixed
parametric derivatives are controlled by Cauchy estimates
(\cref{thm:cauchy}), and the resulting rates are independent of the
number of activated parameters. Quadratic observables such as the radar cross sections of Ref.~\refcite{EscapilJerez2024} are not holomorphic under this
complexification, but remain accessible through \cref{rem:rcs}.
Classical shape differentiability of electromagnetic boundary integral
operators was studied by Costabel and Le Lou\"er~\cite{CostabelLeLouerI,CostabelLeLouerII}; their smooth-boundary
pseudo-homogeneous calculus is not used here.

Finally, the closest reference to our present contribution is Ref.~\refcite{DolzHenriquez2024} for weakly singular scalar kernels
on affinely parametrized Lipschitz boundaries. Therein operators act in
\(L^2\), and kernel estimates do not provide the fractional
smoothing that the Maxwell trace norm requires. The
obstruction is qualitative rather than quantitative. The estimate
available on an admissible set, \cref{lem:kernel}, is pointwise: it
dominates the deformed kernel by \(\abs{\xhat-\yhat}^{-1}\) and thereby
bounds \(\Vop_{\bz}\) between \(L^2\)-based endpoints, but it registers
no Sobolev gain at either of them. Interpolation redistributes the
smoothing present at the endpoints; it cannot create smoothing that is
absent from them. What the graph norm of \(\Hdiv{\Gamma}\) demands is a
gain of one full order, namely
\(\Vop_{\bz}:H^{-1/2}(\Ghat)\to H^{1/2}(\Ghat)\). On a real surface that
order is supplied by the elliptic problem the single layer solves; at a
complex parameter \(\rr_{\bz}(\Ghat)\) is not a surface, so that route
is unavailable and the gain has to be produced by other means. Supplying
it is the analytical content of \cref{thm:scalar}. Related operator-valued shape-holomorphy results have been
obtained for the Laplace Calder\'on projector in two dimensions~\cite{HenriquezSchwab2021} and for boundary integral operators on multiple open arcs~\cite{PintoHenriquezJerez2024}; none of those
settings faces the geometry-dependent Maxwell energy space treated
here. Our deformation amplitudes are measured in \(W^{2,\infty}\)
rather than in the weaker norms of Refs.~\refcite{DolzHenriquez2024} and~\refcite{JerezSchwabZech2017}. This is the price of the
fractional gain: the matrix weights \(F_{\bz}\) must act as uniformly
bounded multipliers on \(H^s\) for \(\abs s\le1\), and the ambient
realization of \cref{sec:newton-route} is quantitative. The \(C^{1,1}\)
hypothesis enters only there and in the componentwise identifications
of \cref{lem:tangential-identification,lem:piola}, not in the
well-posedness of the EFIE itself; see \cref{sec:open}.

\subsection{Proof strategy and organization}

The surface Piola map \(\Pop_{\by}:X(\Ghat)\to X(\Gamma_{\by})\)
satisfies
\begin{equation}\label{eq:intro-piola}
(\Pop_{\by}\widehat\uu)\circ\rr_{\by}=J_{\by}^{-1}F_{\by}\widehat\uu,
\qquad
(\divy\Pop_{\by}\widehat\uu)\circ\rr_{\by}
=J_{\by}^{-1}\divh\widehat\uu,
\end{equation}
where \(F_{\by}=D_{\Ghat}\rr_{\by}\) and \(J_{\by}\) is the surface
Jacobian. In the EFIE bilinear form the factors \(J_{\by}^{-1}\)
cancel exactly against the transformed surface measures. Consequently,
the pulled-back form depends on the geometry only through the deformed
Helmholtz kernel and the tangential differential \(F_{\by}\); no normal
field, inverse surface metric or explicit surface Jacobian remains.

This cancellation isolates the analytical issue. Schematically,
\begin{equation}\label{eq:schematic}
\widehat a_{\bz}(\widehat\jj,\widehat\vv)
\ \sim\
\pair{\Vop_{\bz}(F_{\bz}\widehat\jj)}{F_{\bz}\widehat\vv}
-\kappa^{-2}
\pair{\Vop_{\bz}\divh\widehat\jj}{\divh\widehat\vv}.
\end{equation}
Both components of the \(\Hdivgen\) graph norm have
\(H^{-1/2}\) regularity, so boundedness requires
\(\Vop_{\bz}:H^{-1/2}\to H^{1/2}\). For real geometries this is the classical single-layer mapping property~\cite{Costabel1988}. For
complex deformations there is no physical elliptic problem on the
complex surface. We instead realize the Laplace principal part through
a uniformly sectorial divergence-form problem on a fixed real ambient
space and recover the prescribed complex boundary kernel by analytic
continuation. The mechanism is that the surface smoothing \emph{is} the
volume energy isomorphism \(\dot H^{-1}(\R^3)\to\dot H^{1}(\R^3)\)
followed by the trace. After the Piola transformation,
\(\divh\widehat\jj\in H^{-1/2}(\Ghat)\) enters as a fixed
distributional source, and the energy isomorphism and the trace return
precisely the \(H^{1/2}(\Ghat)\) regularity needed to pair against
\(\divh\widehat\vv\in H^{-1/2}(\Ghat)\) in \eqref{eq:schematic}. The
\(C^{1,1}\) hypothesis is used in this ambient realization; possible
lower-regularity extensions are discussed in \cref{sec:open}.

\Cref{sec:geometry,sec:maxwell} set up the real and complex geometry,
the admissible parameter domains, the Maxwell trace spaces, the surface
Piola map and the EFIE, and \cref{sec:pullback} proves the exact
pullback formula. The fractional single-layer theorem is established in
\cref{sec:scalarSL} and used in \cref{sec:operator} to prove operator
holomorphy and, under nonresonance, holomorphy of the current.
\Cref{sec:consequences,sec:galerkin} derive the sparse-approximation,
far-field and fixed-reference Galerkin consequences, and
\cref{sec:open} discusses extensions.

\section{Geometry and complex-parametric admissibility}
\label{sec:geometry}

\subsection{Reference surface}

Let $\widehat D\subset\R^3$ be a bounded Lipschitz domain with connected
boundary $\Ghat:=\partial\widehat D$ and outward unit normal $\nn$.

\begin{assumption}[Reference geometry]\label{ass:reference}
The surface $\Ghat$ is compact, orientable, and of class $C^{1,1}$. It
admits a finite atlas
$\chi_\ell:\omega_\ell\subset\R^2\to U_\ell\subset\Ghat$,
$\ell=1,\dots,L$, whose $C^{1,1}$ constants and the constants in the
associated bi-Lipschitz equivalences are uniformly bounded, and a
subordinate partition of unity $\{\eta_\ell\}$ with
$\eta_\ell\in C^{0,1}_c(U_\ell)$.
\end{assumption}

Since $\Ghat$ is compact and of class $C^{1,1}$, the geodesic distance
$d_{\Ghat}$ and the ambient Euclidean distance are uniformly equivalent
on $\Ghat$,
\begin{equation}\label{eq:distcomp}
\abs{\xhat-\yhat}
\le
d_{\Ghat}(\xhat,\yhat)
\le
c_{\Ghat}\abs{\xhat-\yhat},
\qquad \xhat,\yhat\in\Ghat .
\end{equation}
We use this equivalence throughout. The $C^{1,1}$ hypothesis enters explicitly in \cref{sec:scalarSL} and
in the componentwise Sobolev-space arguments of
\cref{lem:tangential-identification} and \cref{lem:piola}. The latter
two uses are not intrinsic; see \cref{sec:open}.

\subsection{Affine-parametric surface maps}

Let $\ps_j\in W^{2,\infty}(\Ghat;\R^3)$, $j\ge1$, and set
\begin{equation}\label{eq:bj}
b_j:=\norm{\ps_j}_{W^{2,\infty}(\Ghat;\R^3)} .
\end{equation}
A coordinate with $b_j=0$ has $\ps_j\equiv0$ and does not act on the
geometry; discarding such coordinates, we assume throughout that
$b_j>0$ for every $j$, so that the divisions by $b_j$ appearing in the
choice of Cauchy radii below are legitimate.
With $\U=[-1,1]^\N$ define
\begin{equation}\label{eq:ry}
\rr_{\by}=\Id+\sum_{j\ge1}y_j\ps_j,
\qquad
\Gamma_{\by}:=\rr_{\by}(\Ghat).
\end{equation}
Since $\rr_{\by}$ is injective and $\Gamma_{\by}$ is a closed embedded
surface, we write $D_{\by}$ for the bounded domain enclosed by
$\Gamma_{\by}$; no extension of $\rr_{\by}$ off $\Ghat$ is needed, and
the resonance condition of \cref{ass:invert} is meaningful without
introducing a volume deformation.
Throughout, $\U$ carries the product topology, with respect to which it
is compact by Tychonoff's theorem. We follow the affine-parametric framework of Refs.~\refcite{JerezSchwabZech2017},~\refcite{AylwinJerezSchwabZech2020} and~\refcite{DolzHenriquez2024}.

\begin{assumption}[Parametric geometry]\label{ass:param}
There exists $0<p<1$ with $\bb=(b_j)_{j\ge1}\in\ell^p(\N)$. Moreover
the family $\{\rr_{\by}:\by\in\U\}$ consists of injective maps and is
uniformly bi-Lipschitz and uniformly $C^{1,1}$: there are constants
$0<c_{\rm lip}\le C_{\rm lip}$ and $C_{1,1}>0$ such that
\begin{equation}\label{eq:bilip}
c_{\rm lip}\abs{\xhat-\yhat}
\le
\abs{\rr_{\by}(\xhat)-\rr_{\by}(\yhat)}
\le
C_{\rm lip}\abs{\xhat-\yhat}
\qquad
\forall\,\xhat,\yhat\in\Ghat,\ \forall\,\by\in\U,
\end{equation}
and
$\sup_{\by\in\U}\norm{\rr_{\by}}_{C^{1,1}(\Ghat)}\le C_{1,1}$.
\end{assumption}

Since $p<1$ we have $\ell^p(\N)\subset\ell^1(\N)$, so \eqref{eq:ry}
converges absolutely in $W^{2,\infty}(\Ghat;\R^3)$ and $\rr_{\by}$ is
well defined for every $\by\in\U$. The conditions in \cref{ass:param} are not all
independent: $\bb\in\ell^p(\N)\subset\ell^1(\N)$ already gives uniform
$W^{2,\infty}$ bounds, while \eqref{eq:bilip} adds global injectivity
and a uniform lower bi-Lipschitz bound. Conditions of the type
\begin{equation}\label{eq:sufficient-small}
\norm{\textstyle\sum_{j\ge1}\abs{D_{\Ghat}\ps_j}}_{L^\infty(\Ghat)}<1 ,
\end{equation}
as in Eq.~(5.2) of Ref.~\refcite{JerezSchwabZech2017} and \eqref{eq:JSZ-assumption},
make $\Id+D_{\Ghat}\sum_jy_j\ps_j$ uniformly nonsingular and, together
with a quantitative global embedding hypothesis, are standard sufficient
conditions for \eqref{eq:bilip}. Derivative smallness alone yields only
a uniform immersion and hence local invertibility; the quantitative
lower bi-Lipschitz constant in \eqref{eq:bilip} is a separate
co-Lipschitz requirement. We work directly with the intrinsic
hypothesis \eqref{eq:bilip}, which is all that is used below.

Let $F_{\by}(\xhat):=D_{\Ghat}\rr_{\by}(\xhat)$ be the tangential
differential. For complex parameters there is no physical tangent plane
at $\rr_{\bz}(\xhat)$, so we regard $F_{\bz}(\xhat)$ throughout as a
bundle map $T_{\xhat}\Ghat\to\C^3$, whose range for real $\bz=\by$ is
$T_{\rr_{\by}(\xhat)}\Gamma_{\by}$. Norms such as
$\norm{F_{\bz}}_{W^{1,\infty}(\Ghat;\C^{3\times2})}$ refer to the
matrix representation in the fixed finite atlas of
\cref{ass:reference}, uniformly over its charts. Let the surface
Jacobian $J_{\by}$ be defined by
$\ds_{\rr_{\by}(\xhat)}=J_{\by}(\xhat)\ds_{\xhat}$.
\Cref{ass:param} gives constants $0<c_J\le C_J$ with
\begin{equation}\label{eq:Jreal}
c_J\le J_{\by}(\xhat)\le C_J
\qquad\text{for a.e. }\xhat\in\Ghat,\ \text{uniformly in }\by\in\U .
\end{equation}

\Cref{ass:param} is satisfiable. Let $\{\ps_j\}_{j\ge1}\subset
W^{2,\infty}(\Ghat;\R^3)$ be any family whose amplitudes satisfy
$\bb\in\ell^p(\N)$ for some $0<p<1$; replacing each $\ps_j$ by
$\theta\ps_j$ scales $\bb$ by $\theta>0$ and preserves that summability.
Since $\abs{\ps_j(\xhat)-\ps_j(\yhat)}\le c_{\Ghat}b_j\abs{\xhat-\yhat}$
by \eqref{eq:distcomp}, choosing $\theta$ so small that
$c_{\Ghat}\sum_{j\ge1}b_j<1$ makes every $\rr_{\by}$, $\by\in\U$,
injective and bi-Lipschitz with constants independent of $\by$, while
the uniform $C^{1,1}$ bound follows from $\bb\in\ell^1(\N)$.

\subsection{Admissible complex domains}

We distinguish two families of complex parameter domains. For $\rho>1$ set
\begin{equation}\label{eq:tube}
\mathcal O_\rho
:=
\{z\in\C:\operatorname{dist}(z,[-1,1])<\rho-1\},
\qquad
\mathcal O_{\brho}:=\prod_{j\ge1}\mathcal O_{\rho_j},
\end{equation}
and let $\mathcal E_\rho\subset\C$ denote the Bernstein ellipse with
foci $\{-1,1\}$ and semiaxis sum $\rho$, with
$\mathcal E_{\brho}:=\prod_{j\ge1}\mathcal E_{\rho_j}$. The tubes
$\mathcal O_{\brho}$ are the domains on which all geometric extensions
below are constructed. The polyellipses $\mathcal E_{\brho}$ are the
sets on which uniform bounds are recorded, being the domains required
by the polynomial approximation results of Refs.~\refcite{CohenDeVoreSchwab2011} and~\refcite{ChkifaCohenSchwab2015}. The two are
compatible:
\begin{equation}\label{eq:EinO}
\mathcal E_\rho\subset\mathcal O_\rho
\qquad\text{for every }\rho>1 .
\end{equation}
Indeed, every point of $\mathcal E_\rho$ lies within
$\tfrac12(\rho-\rho^{-1})=(\rho-1)\tfrac{\rho+1}{2\rho}<\rho-1$ of
$[-1,1]$; the elementary verification is in \cref{app:ellipse}.

\begin{definition}[Admissible radii]\label{def:admissible}
For $\eps>0$, a sequence $\brho$ with $\rho_j>1$ is
$(\bb,\eps)$-admissible if
\begin{equation}\label{eq:admbudget}
\sum_{j\ge1}(\rho_j-1)b_j\le\eps .
\end{equation}
\end{definition}

This is the admissibility budget used in Refs.~\refcite{JerezSchwabZech2017} and~\refcite{DolzHenriquez2024}. For
$\bz\in\mathcal O_{\brho}$ we set
\begin{equation}\label{eq:rz}
\rr_{\bz}=\Id+\sum_{j\ge1}z_j\ps_j ,
\end{equation}
a $\C^3$-valued map which is \emph{not} interpreted as a physical
surface. For the complexified boundary geometry we use the
complex-bilinear Euclidean product $a\cdot b=\sum_{m=1}^3a_mb_m$ for
$a,b\in\C^3$: no conjugation is applied in the boundary kernels or the
boundary dualities. The auxiliary volume realization of
\cref{sec:newton-route} uses the sesquilinear energy pairing instead,
for coercivity; see \cref{rem:sesqui}.

\begin{lemma}[Complex separation]\label{lem:complex-separation}
There exist $\eps_0>0$ and $c_0>0$, depending only on the uniform
geometric constants of \cref{ass:reference,ass:param}, such that for every $(\bb,\eps_0)$-admissible
$\brho$, every $\bz\in\mathcal O_{\brho}$, and all $\xhat\ne\yhat$ in
$\Ghat$,
\begin{equation}\label{eq:qpositive}
\Re\Big(
\big[\rr_{\bz}(\xhat)-\rr_{\bz}(\yhat)\big]
\cdot
\big[\rr_{\bz}(\xhat)-\rr_{\bz}(\yhat)\big]
\Big)
\ge
c_0\abs{\xhat-\yhat}^2 .
\end{equation}
Moreover
\begin{equation}\label{eq:Fcomplex}
\sup_{\bz\in\mathcal O_{\brho}}
\norm{F_{\bz}}_{W^{1,\infty}(\Ghat;\C^{3\times2})}
\le C_F
\end{equation}
uniformly over $(\bb,\eps_0)$-admissible $\brho$, and there exists
$\by\in\U$, depending on $\bz$, with
\begin{equation}\label{eq:Fclose}
\norm{F_{\bz}-F_{\by}}_{W^{1,\infty}(\Ghat)}\le \eps_0 .
\end{equation}
\end{lemma}

\begin{proof}
Let $\bz\in\mathcal O_{\brho}$. By the definition \eqref{eq:tube} of
the tube we may choose $\by\in\U$ coordinatewise so that
$\abs{z_j-y_j}<\rho_j-1$ for every $j$. Writing $\bw:=\bz-\by$, the
admissibility condition \eqref{eq:admbudget} gives
\begin{equation}\label{eq:wbudget}
\sum_{j\ge1}\abs{w_j}b_j
\le
\sum_{j\ge1}(\rho_j-1)b_j
\le\eps_0 .
\end{equation}
This is the only step in which the geometry of the complex parameter
domain is used; it explains the use of $\mathcal O_{\brho}$, rather
than $\mathcal E_{\brho}$, as the extension domain.
Put
\[
a:=\rr_{\by}(\xhat)-\rr_{\by}(\yhat)\in\R^3,
\qquad
\bdelta:=\sum_{j\ge1}w_j\big[\ps_j(\xhat)-\ps_j(\yhat)\big]\in\C^3 .
\]
By \eqref{eq:bilip}, $\abs{a}\ge c_{\rm lip}\abs{\xhat-\yhat}$, while
$\abs{\ps_j(\xhat)-\ps_j(\yhat)}\le b_j\,d_{\Ghat}(\xhat,\yhat)
\le c_{\Ghat}b_j\abs{\xhat-\yhat}$ by \eqref{eq:distcomp}, so that
\eqref{eq:wbudget} gives
$\abs{\bdelta}\le c_{\Ghat}\eps_0\abs{\xhat-\yhat}$. Since $a$ is real,
\[
\Re\big[(a+\bdelta)\cdot(a+\bdelta)\big]
\ge
\abs{a}^2-2\abs{a}\abs{\bdelta}-\abs{\bdelta}^2
\ge
\Big(c_{\rm lip}^2
-2c_{\Ghat}C_{\rm lip}\eps_0
-c_{\Ghat}^2\eps_0^2\Big)\abs{\xhat-\yhat}^2 ,
\]
and \eqref{eq:qpositive} follows with $c_0=c_{\rm lip}^2/2$ once
$\eps_0$ is small enough. For \eqref{eq:Fcomplex} and
\eqref{eq:Fclose}, note that
$F_{\bz}-F_{\by}=\sum_{j}w_jD_{\Ghat}\ps_j$, so by \eqref{eq:wbudget}
$\norm{F_{\bz}-F_{\by}}_{W^{1,\infty}}\le\sum_j\abs{w_j}b_j\le\eps_0$,
while $\norm{F_{\by}}_{W^{1,\infty}}\le C_{1,1}$ uniformly in
$\by\in\U$.
\end{proof}

Define
\begin{equation}\label{eq:qz}
q_{\bz}(\xhat,\yhat)
:=
\big[\rr_{\bz}(\xhat)-\rr_{\bz}(\yhat)\big]
\cdot
\big[\rr_{\bz}(\xhat)-\rr_{\bz}(\yhat)\big] .
\end{equation}
By \cref{lem:complex-separation}, $q_{\bz}(\xhat,\yhat)$ lies in the
open right half-plane whenever $\xhat\ne\yhat$. We may therefore use
the principal branch of the square root and set
\begin{equation}\label{eq:Gz}
G_{\kappa,\bz}(\xhat,\yhat)
:=
\frac{\exp\big(\mathrm i\kappa\sqrt{q_{\bz}(\xhat,\yhat)}\,\big)}
{4\pi\sqrt{q_{\bz}(\xhat,\yhat)}} ,
\end{equation}
which for real $\bz=\by\in\U$ reduces to
$G_\kappa(\rr_{\by}(\xhat),\rr_{\by}(\yhat))$.

The following bound is Lemma~2.11 of Ref.~\refcite{DolzHenriquez2024};
we include a proof for completeness.

\begin{lemma}[Weak singularity and coordinate holomorphy]
\label{lem:kernel}
Let $\brho$ be $(\bb,\eps_0)$-admissible. Then
\begin{equation}\label{eq:kernelbound}
\abs{G_{\kappa,\bz}(\xhat,\yhat)}
\le
\frac{C_\kappa}{\abs{\xhat-\yhat}},
\qquad
\xhat\ne\yhat,
\end{equation}
uniformly in $\bz\in\mathcal O_{\brho}$, with $C_\kappa$ depending only
on $\kappa$, $c_0$, $C_{\rm lip}$, $\eps_0$ and $\operatorname{diam}
\Ghat$. For fixed $\xhat\ne\yhat$ the map
$\bz\mapsto G_{\kappa,\bz}(\xhat,\yhat)$ is holomorphic in every
coordinate and jointly holomorphic on finite-dimensional sections.
\end{lemma}

\begin{proof}
Write $q=q_{\bz}(\xhat,\yhat)$. By \eqref{eq:qpositive},
$\Re q\ge c_0\abs{\xhat-\yhat}^2$, hence
$\abs{q}\ge c_0\abs{\xhat-\yhat}^2$ and
$\abs{\sqrt q}\ge\sqrt{c_0}\abs{\xhat-\yhat}$. In the notation of the
proof of \cref{lem:complex-separation},
$\abs{q}\le(\abs{a}+\abs{\bdelta})^2\le
(C_{\rm lip}+c_{\Ghat}\eps_0)^2\abs{\xhat-\yhat}^2$, so
$\abs{\Im\sqrt q}\le\abs{\sqrt q}\le
C\abs{\xhat-\yhat}\le C\operatorname{diam}\Ghat$ and therefore
$\abs{\exp(\mathrm i\kappa\sqrt q)}
=\exp(-\kappa\Im\sqrt q)\le e^{\kappa C\operatorname{diam}\Ghat}$.
Combining the two bounds gives \eqref{eq:kernelbound}. Holomorphy
follows from the affine dependence of $\rr_{\bz}$ on $\bz$, the
polynomial expression \eqref{eq:qz}, the fact that $q_{\bz}$ avoids the
branch cut, and holomorphic composition.
\end{proof}

\begin{remark}\label{rem:eps-star}
Each statement below carries its own admissibility threshold, every such
threshold is at most $\eps_0$, and \cref{lem:complex-separation} is
therefore available wherever it is invoked. Every holomorphic extension
is understood on the corresponding admissible tube
$\mathcal O_{\brho}$, with uniform approximation bounds recorded on
$\mathcal E_{\brho}$.
\end{remark}

\section{Maxwell traces, the EFIE, and the surface Piola map}
\label{sec:maxwell}

\subsection{Trace spaces and duality}

Let $\Gamma$ be a compact orientable Lipschitz boundary with outward
unit normal $\nn$. For $s\in[-1,1]$ let $\Hpar{s}{\Gamma}$ denote the
space of tangential fields whose Cartesian components lie in
$H^s(\Gamma)$, and set
\[
X(\Gamma):=\Hdiv{\Gamma}
=\big\{\uu\in\Hpar{-1/2}{\Gamma}:\divg\uu\in H^{-1/2}(\Gamma)\big\},
\]
and $Y(\Gamma):=\Hcurl{\Gamma}$,
the latter defined analogously with $\curlg$ in place of $\divg$, with
graph norms
\begin{equation}\label{eq:Xnorm}
\norm{\uu}_{X(\Gamma)}^2
=
\norm{\uu}_{\Hpar{-1/2}{\Gamma}}^2
+
\norm{\divg\uu}_{H^{-1/2}(\Gamma)}^2 .
\end{equation}

\begin{lemma}[Tangential spaces on $C^{1,1}$ surfaces]
\label{lem:tangential-identification}
Let $\Gamma$ be of class $C^{1,1}$ and $\abs{s}\le1$. Then
$\Hpar{s}{\Gamma}$ coincides, with equivalent norms, with the range of
the tangential projector $\Pi_{\Gamma}\uu=\uu-(\uu\cdot\nn)\nn$ acting
on $H^s(\Gamma;\C^3)$, and $\Pi_\Gamma$ is bounded on
$H^s(\Gamma;\C^3)$.
\end{lemma}

\begin{proof}
For a $C^{1,1}$ surface, $\nn\in W^{1,\infty}(\Gamma;\R^3)$, hence
$\uu\mapsto(\uu\cdot\nn)\nn$ is multiplication by a $W^{1,\infty}$
matrix field, which is bounded on $H^s(\Gamma)$ for $0\le s\le1$ by
interpolation between $L^2$ and $H^1$ and, for $-1\le s<0$, by duality.
\end{proof}

\Cref{lem:tangential-identification} justifies the componentwise
application of scalar operators to tangential fields in
\cref{sec:scalarSL}; on merely Lipschitz surfaces the identification
fails and the spaces $\Hpar{\pm1/2}{\Gamma}$ must be defined
intrinsically as in Buffa, Costabel and Sheen~\cite{BuffaCostabelSheen2002}. This is another point at which \cref{ass:reference} is used; see \cref{sec:open}.

The tangential trace $\gamma_D\EE:=\nn\times(\EE|_\Gamma\times\nn)$ is
continuous and surjective from $\boldsymbol H(\curl,D)$ onto $Y(\Gamma)$.
The twisted trace $\gamma_\times\EE:=\nn\times\EE|_\Gamma$ is
continuous and surjective onto $X(\Gamma)$~\cite{BuffaCostabelSheen2002,BuffaCostabelSchwab2002,BuffaHiptmair2003,Monk}.
The two spaces are in duality. The $L^2$ pairing of smooth tangential
fields,
\begin{equation}\label{eq:taupairing}
\pairt{\uu}{\vv}:=\int_\Gamma\uu\cdot\vv\ds ,
\end{equation}
extends to a continuous, nondegenerate, complex-bilinear pairing on
$Y(\Gamma)\times X(\Gamma)$ which realizes $Y(\Gamma)$ as the
(complex-linear) dual $X(\Gamma)'$. We denote by
\begin{equation}\label{eq:Idual}
\Iop_\Gamma:Y(\Gamma)\xrightarrow{\ \sim\ }X(\Gamma)'
\end{equation}
the resulting topological isomorphism, the norms being equivalent but
not, with the conventions above, equal. The rotation
$\Rop_\Gamma:\vv\mapsto\vv\times\nn$ is an isomorphism
$Y(\Gamma)\to X(\Gamma)$. Our conventions agree with those of Refs.~\refcite{JerezSchwab2017} and~\refcite{EscapilJerez2024} up to trace rotations. For the
reference boundary we abbreviate $\Xhat:=X(\Ghat)$ and
$\widehat Y:=Y(\Ghat)$.

\begin{remark}[Bilinear, not sesquilinear]\label{rem:bilinear}
The pairing \eqref{eq:taupairing} carries no complex conjugation. As observed in Ref.~\refcite{JerezSchwabZech2017}, conjugating the test
function introduces antiholomorphic dependence and therefore precludes
complex Fr\'echet differentiability. All dualities in this paper are
complex bilinear. If a Hilbert space formulation is written
sesquilinearly, one passes to the linear-dual version by composing the
antilinear dual with complex conjugation; invertibility is unaffected.
Throughout we identify bounded complex-bilinear forms on $X$ with
elements of $\cL(X,X')$ by
\[
\pair{\Phi(b)u}{v}=b(u,v) .
\]
This identification is isometric. Hence a family of forms is
holomorphic if and only if the associated operator family is. We prove
\cref{thm:operator} at the level of forms.
\end{remark}

\subsection{Surface Piola map}

For $\by\in\U$ define the surface contravariant Piola push-forward by
\begin{equation}\label{eq:piola}
(\Pop_{\by}\widehat\uu)\circ\rr_{\by}
=
J_{\by}^{-1}F_{\by}\widehat\uu .
\end{equation}
Its extension to $\Hdiv{\Gamma}$ and the commuting property
\eqref{eq:divcommute} below are consistent with the trace-space
transformation results of Refs.~\refcite{BuffaCostabelSheen2002} and~\refcite{BuffaHiptmair2003}. The additional point
established here is uniformity of the bounds over $\by\in\U$, under the
stronger $C^{1,1}$ hypotheses of \cref{ass:param}.

\begin{lemma}[Piola isomorphism]\label{lem:piola}
Under \cref{ass:reference,ass:param}, the map \eqref{eq:piola} extends
to an isomorphism $\Pop_{\by}:\Xhat\to X(\Gamma_{\by})$ satisfying
\begin{equation}\label{eq:divcommute}
\big(\divy\Pop_{\by}\widehat\uu\big)\circ\rr_{\by}
=
J_{\by}^{-1}\divh\widehat\uu
\qquad\text{in }H^{-1/2}(\Ghat) .
\end{equation}
Moreover there is $C_\Pop\ge1$, independent of $\by$, with
\begin{equation}\label{eq:piolabound}
\sup_{\by\in\U}\Big(
\norm{\Pop_{\by}}_{\cL(\Xhat,X(\Gamma_{\by}))}
+\norm{\Pop_{\by}^{-1}}_{\cL(X(\Gamma_{\by}),\Xhat)}\Big)
\le C_\Pop .
\end{equation}
\end{lemma}

\begin{proof}
For $\widehat\uu\in C^{0,1}$ tangential, \eqref{eq:divcommute} is the
classical surface Piola identity, obtained from the definition of
$\divy$ by duality against $\phi\in C^{0,1}(\Gamma_{\by})$. Set
$\widehat\phi:=\phi\circ\rr_{\by}$. Then
\[
\nabla_{\Gamma_{\by}}\phi\circ\rr_{\by}
=F_{\by}^{-\top}\nabla_{\Ghat}\widehat\phi ,
\]
where $F_{\by}^{-\top}$ denotes the inverse transpose between the
corresponding tangent spaces, represented in local bases by the inverse
transpose of the $2\times2$ tangential Jacobian. Consequently
\begin{align*}
-\int_{\Gamma_{\by}}\Pop_{\by}\widehat\uu\cdot
\nabla_{\Gamma_{\by}}\phi\ds
&=-\int_{\Ghat}J_{\by}^{-1}F_{\by}\widehat\uu\cdot
F_{\by}^{-\top}\nabla_{\Ghat}\widehat\phi\,J_{\by}\ds
\\
&=-\int_{\Ghat}\widehat\uu\cdot\nabla_{\Ghat}\widehat\phi\ds
=\int_{\Ghat}\divh\widehat\uu\,\widehat\phi\ds ,
\end{align*}
which is \eqref{eq:divcommute} after undoing the change of variables.
For the mapping property, multiplication by $F_{\by}$ is uniformly
bounded on $H^s$, $\abs{s}\le1$, because $F_{\by}\in W^{1,\infty}$
uniformly by \cref{ass:param}. Likewise
$J_{\by}^{-1}\in W^{1,\infty}$ uniformly, by \eqref{eq:Jreal} and the
uniform $C^{1,1}$ regularity. Composition with $\rr_{\by}$ is uniformly
bounded on $H^s$, since the maps $\rr_{\by}$ are uniformly bi-Lipschitz
and uniformly $C^{1,1}$. \Cref{lem:tangential-identification} therefore
gives uniform boundedness of
\[
\Pop_{\by}:\Hpar{-1/2}{\Ghat}\to\Hpar{-1/2}{\Gamma_{\by}} .
\]
The divergence identity \eqref{eq:divcommute}, together with the same
argument applied to the scalar $\divh\widehat\uu\in H^{-1/2}$, controls
the second term in \eqref{eq:Xnorm}. The inverse of $\Pop_{\by}$ is the contravariant surface Piola
transform associated with $\rr_{\by}^{-1}:\Gamma_{\by}\to\Ghat$, and is
estimated identically. Finally, the identity \eqref{eq:divcommute} and the bounds extend from
Lipschitz tangential fields to all of $\Xhat$ by density. On the
$C^{1,1}$ surfaces of \cref{ass:reference}, the density of smooth
tangential fields in $\Hdiv{\Ghat}$ is part of the trace-space theory of Refs.~\refcite{BuffaCostabelSheen2002},~\refcite{BuffaCostabelSchwab2002} and~\refcite{BuffaHiptmair2003}.
\end{proof}

The role of $\Pop_{\by}$ is analogous to that of the curl-conforming
volume pullback $\uu\mapsto DT^\top(\uu\circ T)$ of Refs.~\refcite{JerezSchwabZech2017},~\refcite{AylwinJerezSchwabZech2020} and~\refcite{AylwinJerez2023}, but now on the boundary and adapted to the surface divergence.

\subsection{Perfect conductor scattering and the EFIE}

Let $D\subset\R^3$ be a bounded perfect conductor with boundary $\Gamma$
and exterior $D^c$, and let $\kappa>0$ be the exterior wavenumber, with a fixed
time-harmonic convention. Given an entire incident
field with $\curl\curl\EE^{\rm inc}-\kappa^2\EE^{\rm inc}=0$, the
scattered field solves
$\curl\curl\EE^{\rm sc}-\kappa^2\EE^{\rm sc}=0$ in $D^c$ together with
the Silver--M\"uller radiation condition and
$\gamma_D(\EE^{\rm sc}+\EE^{\rm inc})=0$ on $\Gamma$.

With the outgoing fundamental solution
$G_\kappa(x,y)=e^{\mathrm i\kappa\abs{x-y}}/(4\pi\abs{x-y})$, the
electric potential of a surface current $\jj\in X(\Gamma)$ is
\begin{equation}\label{eq:electric-potential}
\Et_\kappa\jj(x)
=
\mathrm i\kappa\int_\Gamma G_\kappa(x,y)\jj(y)\ds_y
-\frac1{\mathrm i\kappa}\nabla_x
\int_\Gamma G_\kappa(x,y)\divg\jj(y)\ds_y ,
\qquad x\notin\Gamma,
\end{equation}
which is the convention of Refs.~\refcite{JerezSchwab2017} and~\refcite{EscapilJerez2024} up to trace rotations and time-harmonic signs. The electric field integral
operator is the (continuous, trace-jump-free) tangential trace
\begin{equation}\label{eq:Tdef}
\Top_{\kappa,\Gamma}:=\gamma_D\Et_\kappa
\ \in\ \cL\big(X(\Gamma),Y(\Gamma)\big) ,
\end{equation}
and we set
\begin{equation}\label{eq:Adef}
\Aop_{\kappa,\Gamma}:=\Iop_\Gamma\Top_{\kappa,\Gamma}
\in\cL\big(X(\Gamma),X(\Gamma)'\big) ,
\end{equation}
with $\Iop_\Gamma$ from \eqref{eq:Idual}.

\begin{proposition}[Variational form of the EFIE operator]
\label{prop:efieform}
For all $\jj,\vv\in X(\Gamma)$,
\begin{equation}\label{eq:efie-form}
\begin{aligned}
a_\Gamma(\jj,\vv)
:=\pair{\Aop_{\kappa,\Gamma}\jj}{\vv}
={}&
\mathrm i\kappa\iint_{\Gamma\times\Gamma}
G_\kappa(x,y)\,\jj(y)\cdot\vv(x)\ds_y\ds_x
\\
&+\frac1{\mathrm i\kappa}\iint_{\Gamma\times\Gamma}
G_\kappa(x,y)\,\divg\jj(y)\,\divg\vv(x)\ds_y\ds_x ,
\end{aligned}
\end{equation}
and $a_\Gamma$ is a bounded complex-bilinear form on
$X(\Gamma)\times X(\Gamma)$.
\end{proposition}

\begin{proof}
For smooth tangential $\jj,\vv$, pair \eqref{eq:electric-potential}
with $\vv$ using \eqref{eq:taupairing}. In the second term only the
tangential part of $\nabla_x$ contributes, so with
$u:=\int_\Gamma G_\kappa(\cdot,y)\divg\jj(y)\ds_y$,
\[
-\frac1{\mathrm i\kappa}\int_\Gamma\nabla_xu\cdot\vv\ds_x
=-\frac1{\mathrm i\kappa}\int_\Gamma\nabla_\Gamma u\cdot\vv\ds_x
=+\frac1{\mathrm i\kappa}\int_\Gamma u\,\divg\vv\ds_x ,
\]
the last step being integration by parts on the closed surface
$\Gamma$. This yields \eqref{eq:efie-form}, whose right-hand side is
$\mathrm i$ times the classical EFIE form
$\kappa\iint G_\kappa\jj\cdot\vv-\kappa^{-1}\iint
G_\kappa\divg\jj\divg\vv$. Boundedness follows from the classical
mapping property of the Helmholtz single layer,
$H^{-1/2}(\Gamma)\to H^{1/2}(\Gamma)$~\cite{Costabel1988,McLean,HsiaoWendland,SauterSchwab2011}. For the
$C^{1,1}$ surfaces considered here, \cref{lem:tangential-identification}
identifies $\Hpar{-1/2}{\Gamma}$ with the range of the tangential
projector in $H^{-1/2}(\Gamma;\C^3)$. The scalar single-layer mapping
theorem therefore applies componentwise to $\jj$, and directly to
$\divg\jj$. Together with \eqref{eq:Xnorm} this proves boundedness, and
the identity extends to $X(\Gamma)$ by density.
\end{proof}

\begin{remark}[Sign convention]\label{rem:sign}
The relative $+$ sign between the two terms of \eqref{eq:efie-form}
follows from the surface integration by parts and is consistent with
\eqref{eq:electric-potential}; the same sign is recovered from the
far-field formula \eqref{eq:farphys}, see \cref{sec:farfield}.
\end{remark}

The EFIE therefore reads
\begin{equation}\label{eq:efieabstract}
\Aop_{\kappa,\Gamma}\jj=f_\Gamma
\quad\text{in }X(\Gamma)' ,
\qquad
f_\Gamma:=-\Iop_\Gamma\gamma_D\EE^{\rm inc} ,
\end{equation}
and is uniquely solvable when $\kappa^2$ is not an interior electric
eigenvalue of $D$; see Refs.~\refcite{BuffaHiptmair2003} and~\refcite{Nedelec}, and the boundary integral analysis in Refs.~\refcite{JerezSchwab2017} and~\refcite{EscapilJerez2024}.
At resonant frequencies the direct EFIE loses invertibility even though
the exterior scattering problem remains uniquely solvable.

\section{Exact pullback and Jacobian cancellation}\label{sec:pullback}

For $\by\in\U$ define the pulled-back operator and form by
\begin{equation}\label{eq:pullback-def}
\widehat\Aop_{\by}:=\Pop_{\by}'\,\Aop_{\kappa,\Gamma_{\by}}\,\Pop_{\by}
\in\cL(\Xhat,\Xhat'),
\qquad
\widehat a_{\by}(\widehat\jj,\widehat\vv)
:=\pair{\widehat\Aop_{\by}\widehat\jj}{\widehat\vv}
=a_{\Gamma_{\by}}(\Pop_{\by}\widehat\jj,\Pop_{\by}\widehat\vv),
\end{equation}
$\Pop_{\by}'$ denoting the (complex-linear) dual of $\Pop_{\by}$. By
\cref{lem:piola}, $\widehat\Aop_{\by}$ is an isomorphism if and only if
$\Aop_{\kappa,\Gamma_{\by}}$ is, with
\begin{equation}\label{eq:inverse-transfer}
C_\Pop^{-2}\norm{\Aop_{\kappa,\Gamma_{\by}}^{-1}}
\le
\norm{\widehat\Aop_{\by}^{-1}}
\le
C_\Pop^{2}\norm{\Aop_{\kappa,\Gamma_{\by}}^{-1}} .
\end{equation}

\begin{proposition}[Piola cancellation]\label{prop:cancel}
For all $\widehat\jj,\widehat\vv\in\Xhat$ and all $\by\in\U$,
\begin{equation}\label{eq:fixed-form}
\begin{aligned}
\widehat a_{\by}(\widehat\jj,\widehat\vv)
={}&
\mathrm i\kappa
\iint_{\Ghat\times\Ghat}
G_\kappa\big(\rr_{\by}(\xhat),\rr_{\by}(\yhat)\big)
\big[F_{\by}(\yhat)\widehat\jj(\yhat)\big]
\cdot
\big[F_{\by}(\xhat)\widehat\vv(\xhat)\big]
\ds_{\yhat}\ds_{\xhat}
\\
&+\frac1{\mathrm i\kappa}
\iint_{\Ghat\times\Ghat}
G_\kappa\big(\rr_{\by}(\xhat),\rr_{\by}(\yhat)\big)
\,\divh\widehat\jj(\yhat)\,\divh\widehat\vv(\xhat)
\ds_{\yhat}\ds_{\xhat} .
\end{aligned}
\end{equation}
\end{proposition}

\begin{proof}
It suffices to treat Lipschitz tangential fields and to invoke density
and \cref{prop:efieform}. For the vector-potential term,
\eqref{eq:piola} gives
\[
(\Pop_{\by}\widehat\jj)(\rr_{\by}(\yhat))
=J_{\by}(\yhat)^{-1}F_{\by}(\yhat)\widehat\jj(\yhat),
\]
and likewise for $\widehat\vv$, while the change of variables
$x=\rr_{\by}(\xhat)$, $y=\rr_{\by}(\yhat)$ contributes the factor
$J_{\by}(\xhat)J_{\by}(\yhat)$, which cancels the two inverse
Jacobians. For the scalar-potential term, \eqref{eq:divcommute}
produces one inverse Jacobian for each surface divergence, and the two
surface measures cancel them exactly.
\end{proof}

\begin{remark}[Structure of the pulled-back form]\label{rem:significance}
Neither the unit normal, nor an inverse surface metric, nor a surface
Jacobian appears explicitly in \eqref{eq:fixed-form}: the divergence has
been frozen completely on the reference boundary, and the only
geometric quantities remaining are the deformed kernel and the first
tangential differential $F_{\by}$.
\end{remark}

\section{Complex-deformed scalar single layers}\label{sec:scalarSL}

\subsection{The mapping theorem}\label{sec:scalar-thm}

The additional step beyond the $L^2$ theory of Ref.~\refcite{DolzHenriquez2024} is a fractional Sobolev mapping result,
uniform over admissible complex deformations. For
$\bz\in\mathcal O_{\brho}$ define
\begin{equation}\label{eq:Vzdef}
(\Vop_{\bz}\phi)(\xhat)
:=
\int_{\Ghat}G_{\kappa,\bz}(\xhat,\yhat)\phi(\yhat)\ds_{\yhat} .
\end{equation}

\begin{theorem}[Fractional holomorphy of the deformed single layer]
\label{thm:scalar}
Under \cref{ass:reference,ass:param} there exists $\eps_1>0$ such that
for every $(\bb,\eps_1)$-admissible $\brho$:
\begin{enumerate}[label=\textup{(\roman*)}]
\item $\Vop_{\bz}$ extends uniquely to a bounded operator
$\Vop_{\bz}:H^{-1/2}(\Ghat)\to H^{1/2}(\Ghat)$ for every
$\bz\in\mathcal O_{\brho}$;
\item there is $C_V>0$, independent of $\brho$ and $\bz$, with
$\sup_{\bz\in\mathcal O_{\brho}}
\norm{\Vop_{\bz}}_{\cL(H^{-1/2},H^{1/2})}\le C_V$;
\item the map
$\bz\mapsto\Vop_{\bz}\in\cL(H^{-1/2}(\Ghat),H^{1/2}(\Ghat))$
is holomorphic in every coordinate on $\mathcal O_{\brho}$, jointly
holomorphic on finite-dimensional sections, and holomorphic along
admissible complex lines in the following sense: for every
$\bz\in\C^{\N}$, every direction $\bw$ with
$\sum_jb_j\abs{w_j}<\infty$, and every \emph{open disc}
$D\subset\C$ such that $\bz+\tau\bw\in\mathcal O_{\brho}$ for all
$\tau\in D$, the map $\tau\mapsto\Vop_{\bz+\tau\bw}$ is holomorphic in
the operator norm on $D$.
\end{enumerate}
\end{theorem}

\Cref{thm:scalar} is proved by the fixed-domain trace--Newton
construction of \cref{sec:newton-route}, whose logical spine is
\begin{equation}\label{eq:chain}
\begin{aligned}
\text{complex geometry}
&\ \Longrightarrow\
\text{complex elliptic coefficient field on }\R^3
\\
&\ \Longrightarrow\ \text{Newton potential}
\ \Longrightarrow\ H^{-1/2}\to H^{1/2} ,
\end{aligned}
\end{equation}
and it uses only the complex separation of
\cref{lem:complex-separation}, standard elliptic theory on $\R^3$, and
analytic continuation from the real geometries. Constants depending on $\kappa$ are not tracked explicitly (those of Proposition 5.1 are $\kappa$-independent); we return to this in \cref{sec:open}.


\subsection{Proof by a fixed-domain trace--Newton construction}
\label{sec:newton-route}

For complex $\bz$, the set $\rr_{\bz}(\Ghat)\subset\C^3$ is not
regarded as a physical boundary; all partial differential equations
below are posed on the fixed real space $\R^3$. Near each real
geometry, the surface deformation is extended to a complex perturbation
$\mathcal T_{\by,\bz}$ of a real ambient bi-Lipschitz transformation
$\mathcal T_{\by}$ of $\R^3$.
\Cref{lem:ambient} below provides such a transformation with constants
uniform in $\by$; its proof, which is where the $C^{1,1}$ hypothesis
enters this route, is deferred to \cref{app:ambient}. One then
solves a uniformly coercive divergence-form problem on the fixed real
space with the complexified coefficient
\[
\mathcal A_{\by,\bz}
:=
\big(\det D\mathcal T_{\by,\bz}\big)\,
D\mathcal T_{\by,\bz}^{-1}D\mathcal T_{\by,\bz}^{-\top} .
\]
Complexification enters only through these coefficients. Writing
$\mathsf L_{\by,\bz}$ for the associated isomorphism
$\dot H^1(\R^3)\to\dot H^{-1}(\R^3)$ and $\widehat\gamma$ for the trace
on $\Ghat$, the smoothing appears as the composition
\begin{equation}\label{eq:newton-chain}
H^{-1/2}(\Ghat)
\xrightarrow{\ \widehat\gamma^*\ }
\dot H^{-1}(\R^3)
\xrightarrow{\ \mathsf L_{\by,\bz}^{-1}\ }
\dot H^{1}(\R^3)
\xrightarrow{\ \widehat\gamma\ }
H^{1/2}(\Ghat) ,
\end{equation}
that is, as the energy isomorphism $\dot H^{-1}\to\dot H^1$ followed by
the trace. \Cref{prop:newton} identifies
$\mathsf W_{\by,\bz}:=\widehat\gamma\mathsf L_{\by,\bz}^{-1}
\widehat\gamma^*$ with the complex-deformed Laplace single layer
$G_{0,\bz}=(4\pi)^{-1}q_{\bz}^{-1/2}$, first at real parameters and
then by analytic continuation. \Cref{lem:helm-corr} treats the
Helmholtz correction $G_{\kappa,\bz}-G_{0,\bz}$. That correction is
regular at the diagonal because
$s\mapsto(e^{\mathrm i\kappa s}-1)/(4\pi s)$ is entire. Together they prove \cref{thm:scalar} (\cref{prop:newton-proof}).


\subsubsection{Ambient realization of the real surface maps}

\begin{lemma}[Uniform ambient realization]\label{lem:ambient}
There are constants $C_{\mathcal T},c_{\mathcal T},C_E>0$ and a ball
$B\subset\R^3$ containing the closure of a uniform tubular
neighbourhood of every intermediate surface $\Gamma_{t\by}$,
$\by\in\U$, $t\in[0,1]$, such that
for each $\by\in\U$ there is an orientation-preserving bi-Lipschitz map
$\mathcal T_{\by}:\R^3\to\R^3$ with
\[
\mathcal T_{\by}|_{\Ghat}=\rr_{\by},
\qquad
\mathcal T_{\by}=\Id\ \text{ on }\R^3\setminus B,
\]
and, almost everywhere,
$\norm{D\mathcal T_{\by}}_{L^\infty}
+\norm{D\mathcal T_{\by}^{-1}}_{L^\infty}\le C_{\mathcal T}$ and
$\det D\mathcal T_{\by}\ge c_{\mathcal T}$. Moreover there is a bounded
linear extension operator
$\mathsf E:W^{1,\infty}(\Ghat;\C^3)\to W^{1,\infty}_0(B;\C^3)$ with
\[
(\mathsf Eh)|_{\Ghat}=h ,
\qquad
\norm{\mathsf Eh}_{W^{1,\infty}(\R^3)}
\le C_E\norm{h}_{W^{1,\infty}(\Ghat)} .
\]
\end{lemma}

The point is not ambient extendability itself, which is routine for a
single deformation, but uniform quantitative control over the whole
family. The realization is global and orientation preserving for every
$\by$, with $C_{\mathcal T}$, $c_{\mathcal T}$ and $C_E$ independent of
$\by$. That is what the trace--Newton construction requires. No
continuity of $\by\mapsto\mathcal T_{\by}$ is asserted or needed. The
construction is given in full in \cref{app:ambient}.

\begin{proof}
The construction is quantitative, and uniformity in $\by$ requires
tracking the geometric constants; the proof is given in
\cref{app:ambient}.
\end{proof}

\subsubsection{Complex perturbation of the ambient realization}

Fix $\by\in\U$ and let $\bz\in\C^{\N}$ satisfy
$\sum_{j\ge1}b_j\abs{z_j-y_j}<\delta$. Define
\begin{equation}\label{eq:ambient-complex}
\mathcal T_{\by,\bz}
:=
\mathcal T_{\by}+\mathsf E\big(\rr_{\bz}-\rr_{\by}\big) ,
\end{equation}
so that $\mathcal T_{\by,\bz}|_{\Ghat}=\rr_{\bz}$ and
$\norm{D\mathcal T_{\by,\bz}-D\mathcal T_{\by}}_{L^\infty}
\le C_E\sum_{j\ge1}b_j\abs{z_j-y_j}$.

The first index $\by$ records only the real ambient realization about
which the complex perturbation is built. The trace of
$\mathcal T_{\by,\bz}$ on $\Ghat$ is $\rr_{\bz}$, independently of that
auxiliary choice, and \cref{prop:newton} will show that the resulting
boundary operator $\mathsf W_{\by,\bz}$ is the intrinsic operator
$\Vop^0_{\bz}$ and hence does not depend on the centre $\by$. Here and
below $\Vop^0_{\bz}$ denotes the complex-deformed \emph{Laplace} single
layer, that is, the operator with kernel
$G_{0,\bz}=(4\pi)^{-1}q_{\bz}^{-1/2}$, which is \eqref{eq:Gz} at
$\kappa=0$; for densities rougher than $C^{0,1}(\Ghat)$ it is
understood by continuous extension, as made precise in
\cref{prop:newton}.

Introduce the matrix class
\begin{equation}\label{eq:Kclass}
\mathcal K
:=
\big\{M\in\R^{3\times3}:
\norm M\le C_{\mathcal T},\
\norm{M^{-1}}\le C_{\mathcal T},\
\det M\ge c_{\mathcal T}\big\} ,
\end{equation}
which is closed and bounded in a finite-dimensional space, hence
compact, and which contains $D\mathcal T_{\by}(x)$ for a.e.\ $x$ and
every $\by\in\U$ by \cref{lem:ambient}. We use $\mathcal K$ rather than
the family $\{D\mathcal T_{\by}(x)\}$ itself, whose compactness is
neither established nor needed: the construction of \cref{lem:ambient}
gives uniform bounds but no continuity of $\by\mapsto\mathcal T_{\by}$.
For $\delta$ small enough, uniformly in $\by$, every matrix
$D\mathcal T_{\by,\bz}(x)$ then lies in a fixed complex neighbourhood
of $\mathcal K$, hence is invertible a.e., and we may set
\begin{equation}\label{eq:Av}
\mathcal J_{\by,\bz}:=\det D\mathcal T_{\by,\bz},
\qquad
\mathcal A_{\by,\bz}
:=
\mathcal J_{\by,\bz}\,
D\mathcal T_{\by,\bz}^{-1}D\mathcal T_{\by,\bz}^{-\top} .
\end{equation}
This is the coefficient produced by the Dirichlet form under the change
of variables: for a real bi-Lipschitz $\mathcal T$ and
$u=v\circ\mathcal T^{-1}$,
\begin{equation}\label{eq:pullback-identity}
\int_{\R^3}\nabla u\cdot\overline{\nabla\phi}\,dx
=
\int_{\R^3}(\det D\mathcal T)\,
D\mathcal T^{-1}D\mathcal T^{-\top}\nabla v
\cdot\overline{\nabla\widehat\phi}\,d\widehat x ,
\qquad
\widehat\phi:=\phi\circ\mathcal T ,
\end{equation}
For real $\mathcal T$ the transformation rule is unchanged by the
conjugation used later in \eqref{eq:Lv}. Formula
\eqref{eq:pullback-identity} fixes the transpose order and the
determinant factor in \eqref{eq:Av}.

\begin{lemma}[Uniform complex ellipticity]\label{lem:volume-elliptic}
After reducing $\delta$ if necessary there are $c_{\rm v},C_{\rm v}>0$,
independent of $\by$ and $\bz$, with
$\norm{\mathcal A_{\by,\bz}}_{L^\infty(\R^3)}\le C_{\rm v}$ and
\begin{equation}\label{eq:volume-sectorial}
\Re\big(\mathcal A_{\by,\bz}(x)\,\xi\cdot\overline{\xi}\big)
\ge
c_{\rm v}\abs{\xi}^2,
\qquad \xi\in\C^3,
\end{equation}
for a.e.\ $x\in\R^3$, the bar denoting complex conjugation. Here
\eqref{eq:volume-sectorial} is the sesquilinear coercivity required by
Lax--Milgram: the unconjugated expression
$\mathcal A\xi\cdot\xi$ is not bounded below on $\C^3$, as
$\mathcal A=I$ and $\xi=\mathrm i\mathbf e_1$ already show. Moreover $\bz\mapsto\mathcal A_{\by,\bz}$ is
holomorphic, in every coordinate, on finite-dimensional sections, and
along admissible complex lines in the sense of
\cref{thm:scalar}\,\textup{(iii)}, with values in
$L^\infty(\R^3;\C^{3\times3})$.
\end{lemma}

\begin{proof}
For $\bz=\by$ the matrix
$\mathcal A_{\by,\by}=(\det D\mathcal T_{\by})
D\mathcal T_{\by}^{-1}D\mathcal T_{\by}^{-\top}$ is real symmetric with
eigenvalues at least
$c_{\mathcal T}C_{\mathcal T}^{-2}$, hence uniformly positive definite.
Write $c_*:=c_{\mathcal T}C_{\mathcal T}^{-2}$, so that
$\mathcal A_{\by,\by}\,\xi\cdot\overline{\xi}\ge c_*\abs\xi^2$ for
$\xi\in\C^3$, the matrix $\mathcal A_{\by,\by}$ being real symmetric. The map $M\mapsto(\det M)M^{-1}M^{-\top}$ is analytic on
$GL(3,\C)$ and Lipschitz on a fixed complex neighbourhood of the
compact class $\mathcal K$ of \eqref{eq:Kclass}, so by
\eqref{eq:ambient-complex}
\[
\norm{\mathcal A_{\by,\bz}-\mathcal A_{\by,\by}}_{L^\infty}
\le
C\sum_{j\ge1}b_j\abs{z_j-y_j} .
\]
Hence
\[
\Re\big(\mathcal A_{\by,\bz}\xi\cdot\overline\xi\big)
=
\mathcal A_{\by,\by}\xi\cdot\overline\xi
+\Re\big((\mathcal A_{\by,\bz}-\mathcal A_{\by,\by})\xi\cdot\overline\xi\big)
\ge
\Big(c_*-C\sum_{j\ge1}b_j\abs{z_j-y_j}\Big)\abs\xi^2 ,
\]
and \eqref{eq:volume-sectorial} follows with $c_{\rm v}:=c_*/2$ upon
taking $\delta<c_*/(2C)$. Since \eqref{eq:ambient-complex} is
affine in $\bz$, the map $\bz\mapsto D\mathcal T_{\by,\bz}$ is
$L^\infty$-valued holomorphic; the determinant is polynomial and the
inverse is given by a locally uniformly convergent Neumann series.

For the directional statement, let $\bw$ satisfy
$\sum_jb_j\abs{w_j}<\infty$ and put
$\ps_{\bw}:=\sum_{j\ge1}w_j\ps_j$, a series converging absolutely in
$W^{2,\infty}(\Ghat;\C^3)$. By \eqref{eq:ambient-complex},
\begin{equation}\label{eq:ambient-line}
\mathcal T_{\by,\bz+\tau\bw}
=
\mathcal T_{\by,\bz}+\tau\,\mathsf E\ps_{\bw} ,
\end{equation}
so $\tau\mapsto D\mathcal T_{\by,\bz+\tau\bw}$ is affine with values in
$L^\infty$, and the same three observations give holomorphy of
$\tau\mapsto\mathcal A_{\by,\bz+\tau\bw}$ on every open disc of
admissible $\tau$, with infinitely many nonzero components of $\bw$
permitted.
\end{proof}

\subsubsection{The Laplace singular part}

Write $G_0(x)=1/(4\pi\abs x)$ and, with $q_{\bz}$ as in
\eqref{eq:qz} and the principal branch of \cref{lem:complex-separation},
\begin{equation}\label{eq:G0z}
G_{0,\bz}(\xhat,\yhat)
:=
\frac1{4\pi\sqrt{q_{\bz}(\xhat,\yhat)}} ,
\end{equation}
which for real $\bz=\by$ is $G_0(\rr_{\by}(\xhat)-\rr_{\by}(\yhat))$.

Let $\dot H^1(\R^3)$ be the completion of $C^\infty_c(\R^3)$ under
$\norm{\nabla\cdot}_{L^2(\R^3)}$, a space carrying no $L^2(\R^3)$ norm.
We use the reference domain $\widehat D$ itself, whose boundary is
exactly $\Ghat$. By the Sobolev embedding
$\dot H^1(\R^3)\hookrightarrow L^6(\R^3)$, valid in dimension three,
\[
\norm u_{L^2(\widehat D)}
\le
\abs{\widehat D}^{1/3}\norm u_{L^6(\widehat D)}
\le
C\norm{\nabla u}_{L^2(\R^3)} ,
\]
so $u|_{\widehat D}\in H^1(\widehat D)$ with
$\norm u_{H^1(\widehat D)}\le C\norm{\nabla u}_{L^2(\R^3)}$. Since
$\partial\widehat D=\Ghat$ is Lipschitz, the classical trace theorem on $\widehat D$ (cf.~Chap.~3 of Ref.~\refcite{McLean}) gives a bounded
$\widehat\gamma:\dot H^1(\R^3)\to H^{1/2}(\Ghat)$. A trace theorem on an arbitrary bounded set
containing $\Ghat$ would only provide a trace on the boundary of that
set, not on the interior hypersurface $\Ghat$.

Let $\dot H^{-1}(\R^3)$ denote the \emph{antidual} of $\dot H^1(\R^3)$
and define $\widehat\gamma^*:H^{-1/2}(\Ghat)\to\dot H^{-1}(\R^3)$
explicitly by
\begin{equation}\label{eq:gammastar}
\pair{\widehat\gamma^*\phi}{v}_{\dot H^{-1},\dot H^1}
:=
\pair{\phi}{\overline{\widehat\gamma v}}_{H^{-1/2},H^{1/2}} ,
\end{equation}
the pairing on the right being the complex-bilinear one used
throughout; for $\phi\in C^{0,1}(\Ghat)$ this reads
$\pair{\widehat\gamma^*\phi}{v}
=\int_{\Ghat}\phi\,\overline{\widehat\gamma v}\ds$. With this
convention $\phi\mapsto\widehat\gamma^*\phi$ is complex \emph{linear},
while $v\mapsto\pair{\widehat\gamma^*\phi}{v}$ is antilinear, as an
element of the antidual must be. Define
$\mathsf L_{\by,\bz}:\dot H^1(\R^3)\to\dot H^{-1}(\R^3)$ by
\begin{equation}\label{eq:Lv}
\pair{\mathsf L_{\by,\bz}u}{v}
:=
\int_{\R^3}\mathcal A_{\by,\bz}\nabla u\cdot\overline{\nabla v}\,dx .
\end{equation}
\begin{remark}[Bilinear boundary duality versus the Hilbert antidual]
\label{rem:sesqui}
Two dualities occur, for different purposes. The Maxwell boundary
formulation uses complex-bilinear duality throughout, as fixed in
\cref{rem:bilinear}; the auxiliary volume problem \eqref{eq:Lv} is
instead posed in the Hilbert antidual, so its test variable is
conjugated and the map
$\widehat\gamma^*:H^{-1/2}(\Ghat)\to\dot H^{-1}(\R^3)$ of
\eqref{eq:gammastar} is complex linear in the density and antilinear in
the test variable, as an adjoint into an antidual must be. It is not to
be confused with the complex-bilinear boundary duality. All parameter
dependence sits in the coefficient $\mathcal A_{\by,\bz}$, which is
never conjugated; consequently $\bz\mapsto\mathsf L_{\by,\bz}^{-1}$ is
a holomorphic operator-valued map and
$\phi\mapsto\mathsf W_{\by,\bz}\phi$ is complex linear.
\Cref{prop:newton} identifies this linear boundary operator with the
complex-deformed single-layer kernel used in the bilinear boundary
formulation, so the objects produced in \eqref{eq:newton-kernel} are
those used everywhere else.
\end{remark}

\begin{proposition}[Pulled-back Newton realization]
\label{prop:newton}
There is $\delta_N>0$ such that, uniformly for $\by\in\U$ and
$\sum_{j\ge1}b_j\abs{z_j-y_j}<\delta_N$, the operator
$\mathsf L_{\by,\bz}$ is an isomorphism
$\dot H^1(\R^3)\to\dot H^{-1}(\R^3)$ with
$\norm{\mathsf L_{\by,\bz}^{-1}}\le C_N$, the operator
\begin{equation}\label{eq:Wnewton}
\mathsf W_{\by,\bz}
:=
\widehat\gamma\,\mathsf L_{\by,\bz}^{-1}\,\widehat\gamma^*
\in\cL\big(H^{-1/2}(\Ghat),H^{1/2}(\Ghat)\big)
\end{equation}
is uniformly bounded, and is holomorphic in $\bz$ in every coordinate,
on finite-dimensional sections, and along admissible complex lines.
Moreover,
\begin{equation}\label{eq:newton-kernel}
(\mathsf W_{\by,\bz}\phi)(\xhat)
=
\int_{\Ghat}G_{0,\bz}(\xhat,\yhat)\phi(\yhat)\ds_{\yhat}
\qquad\text{in }H^{1/2}(\Ghat).
\end{equation}
\end{proposition}

\begin{proof}
By \eqref{eq:volume-sectorial},
$\Re\pair{\mathsf L_{\by,\bz}u}{u}\ge c_{\rm v}\norm u_{\dot H^1}^2$,
so Lax--Milgram gives the inverse with a uniform bound. Since
$\bz\mapsto\mathsf L_{\by,\bz}$ is operator-norm holomorphic by
\cref{lem:volume-elliptic} and inversion is holomorphic on the open set
of isomorphisms, $\bz\mapsto\mathsf L_{\by,\bz}^{-1}$ is holomorphic;
composing with the fixed bounded trace maps gives the statements about
$\mathsf W_{\by,\bz}$. The same reasoning applies along an
admissible complex line: by \eqref{eq:ambient-line} the map
$\tau\mapsto\mathcal A_{\by,\bz+\tau\bw}$ is holomorphic, hence so are
$\tau\mapsto\mathsf L_{\by,\bz+\tau\bw}^{-1}$, by uniform sectoriality
on the disc and holomorphy of inversion, and
$\tau\mapsto\mathsf W_{\by,\bz+\tau\bw}$ in
$\cL(H^{-1/2}(\Ghat),H^{1/2}(\Ghat))$.

For \eqref{eq:newton-kernel} we argue first at real parameters. Let $\widetilde\by\in\U$ satisfy
$C_E\sum_jb_j\abs{\widetilde y_j-y_j}<\tfrac12C_{\mathcal T}^{-1}$.
Then $\mathcal T_{\by,\widetilde\by}$ is a real orientation-preserving
bi-Lipschitz map. It need not coincide with the realization
$\mathcal T_{\widetilde\by}$ that \cref{lem:ambient} supplies at
$\widetilde\by$. This is immaterial. The change-of-variables identity
below requires only that $\mathcal T_{\by,\widetilde\by}$ be a real
bi-Lipschitz ambient map whose boundary trace is
$\rr_{\widetilde\by}$. It is real because $\widetilde\by\in\U$. Moreover
$\mathcal T_{\by}$ is co-Lipschitz, with
$\abs{\mathcal T_{\by}(x)-\mathcal T_{\by}(x')}
\ge C_{\mathcal T}^{-1}\abs{x-x'}$, so
\[
\abs{\mathcal T_{\by,\widetilde\by}(x)-\mathcal T_{\by,\widetilde\by}(x')}
\ge
\Big(C_{\mathcal T}^{-1}
-\operatorname{Lip}\mathsf E(\rr_{\widetilde\by}-\rr_{\by})\Big)
\abs{x-x'}
\ge
\tfrac12C_{\mathcal T}^{-1}\abs{x-x'} ,
\]
which gives global injectivity and the lower bound. Moreover the
determinant is uniformly Lipschitz on a fixed neighbourhood of
$\mathcal K$, so after decreasing $\delta_N$ if necessary
\[
\det D\mathcal T_{\by,\widetilde\by}\ge\tfrac12c_{\mathcal T}
\qquad\text{a.e.},
\]
and $\mathcal T_{\by,\widetilde\by}$ is orientation preserving. Let
$\phi\in C^{0,1}(\Ghat)$, and define the physical
density $\varrho_{\widetilde\by}$ on
$\Gamma_{\widetilde\by}=\rr_{\widetilde\by}(\Ghat)$ by
\begin{equation}\label{eq:phys-density}
\varrho_{\widetilde\by}\circ\rr_{\widetilde\by}
:=
J_{\widetilde\by}^{-1}\phi .
\end{equation}
Let $U_{\widetilde\by}(x)=\int_{\Gamma_{\widetilde\by}}
G_0(x-x')\varrho_{\widetilde\by}(x')\ds_{x'}$ be the corresponding
Laplace single-layer potential, which lies in $\dot H^1(\R^3)$ and
satisfies $-\Delta U_{\widetilde\by}
=\gamma_{\widetilde\by}^*\varrho_{\widetilde\by}$ distributionally. Put
$\widehat U:=U_{\widetilde\by}\circ\mathcal T_{\by,\widetilde\by}$.
Changing variables in the weak formulation gives, for
$v\in\dot H^1(\R^3)$,
\[
\int_{\R^3}\mathcal A_{\by,\widetilde\by}\nabla\widehat U
\cdot\overline{\nabla v}\,dx
=
\int_{\Gamma_{\widetilde\by}}\varrho_{\widetilde\by}\,
\overline{\gamma_{\widetilde\by}
\big(v\circ\mathcal T_{\by,\widetilde\by}^{-1}\big)}\ds .
\]
The transformed distributional source is parameter independent. Indeed,
the surface change of variables
contributes $J_{\widetilde\by}$, which by \eqref{eq:phys-density}
cancels exactly against the factor $J_{\widetilde\by}^{-1}$ carried by
the transported density, leaving
\begin{equation}\label{eq:source-cancel}
\int_{\Gamma_{\widetilde\by}}\varrho_{\widetilde\by}\,
\overline{\gamma_{\widetilde\by}
\big(v\circ\mathcal T_{\by,\widetilde\by}^{-1}\big)}\ds
=
\int_{\Ghat}\phi\,\overline{\widehat\gamma v}\ds ,
\end{equation}
a source independent of the parameter. Hence
$\mathsf L_{\by,\widetilde\by}\widehat U=\widehat\gamma^*\phi$ and, by
uniqueness, $\widehat U=\mathsf L_{\by,\widetilde\by}^{-1}
\widehat\gamma^*\phi$. Taking the trace and using
$\mathcal T_{\by,\widetilde\by}|_{\Ghat}=\rr_{\widetilde\by}$ together
with \eqref{eq:phys-density} once more,
\[
(\widehat\gamma\widehat U)(\xhat)
=
\int_{\Gamma_{\widetilde\by}}
G_0\big(\rr_{\widetilde\by}(\xhat)-x'\big)
\varrho_{\widetilde\by}(x')\ds_{x'}
=
\int_{\Ghat}
G_0\big(\rr_{\widetilde\by}(\xhat)-\rr_{\widetilde\by}(\yhat)\big)
\phi(\yhat)\ds_{\yhat} ,
\]
which is \eqref{eq:newton-kernel} at real parameters.

We pass to complex parameters in two stages. Throughout, fix
$\phi,\psi\in C^{0,1}(\Ghat)$ and note that, by \eqref{eq:qpositive},
\begin{equation}\label{eq:G0-majorant}
\abs{G_{0,\bz}(\xhat,\yhat)}
=
\frac1{4\pi\abs{q_{\bz}(\xhat,\yhat)}^{1/2}}
\le
\frac1{4\pi\sqrt{c_0}}\,\frac1{\abs{\xhat-\yhat}} ,
\end{equation}
since $\abs{q_{\bz}}\ge\Re q_{\bz}\ge c_0\abs{\xhat-\yhat}^2$; this,
and not \eqref{eq:kernelbound}, is the majorant appropriate to
$G_{0,\bz}$.

\emph{Finitely supported parameters.} Suppose first that
$\bz-\by$ has finite support $S$. Restrict to the finite-dimensional
section in which $z_j=y_j$ for every $j\notin S$ and only the
coordinates in $S$ vary. For each $j\in S$, choose a nonempty open interval
$I_j\subset[-1,1]\cap(y_j-\eta_j,y_j+\eta_j)$, with $\eta_j$ small
enough that the resulting real parameters satisfy the smallness
condition of the previous paragraph. If $y_j=\pm1$, take a one-sided interval lying arbitrarily close to the
endpoint, $I_j=(1-\eta_j,1)$ or $I_j=(-1,-1+\eta_j)$. The endpoint
itself plays no role: the argument uses only that $I_j$ is a nonempty
open subinterval of $[-1,1]$.

The box $I_S:=\prod_{j\in S}I_j$ is a nonempty open real box contained
in $\U$ and in the neighbourhood above, where equality has already been
proved, and it has accumulation points in the complex section. The two
functions
\[
\bz\longmapsto\pair{\mathsf W_{\by,\bz}\phi}{\psi},
\qquad
\bz\longmapsto\iint G_{0,\bz}\phi\psi ,
\]
are both holomorphic on the complex section under consideration. The
first is holomorphic by the earlier part of the proof. The second is
holomorphic by \cref{lem:complex-separation} and dominated convergence,
with \eqref{eq:G0-majorant} as majorant. They agree on the real box
$I_S$.
Reducing $\delta_N$ if necessary below the threshold $\eps_0$ of
\cref{lem:complex-separation}, every parameter in the section satisfies
the sectorial estimate \eqref{eq:qpositive} for $q_{\bz}$; the
principal branch defining $G_{0,\bz}$ is therefore holomorphic
throughout the connected domain introduced next, and the continuation
path never leaves the region where that branch is defined.
Introduce the connected finite-dimensional domain
\begin{equation}\label{eq:DS}
D_S:=\Big\{\bm\zeta_S\in\C^S:
\sum_{j\in S}b_j\abs{\zeta_j-y_j}<\delta_N\Big\} ,
\end{equation}
which contains both $I_S$ and the prescribed point $\bz_S$, the latter
strictly inside by hypothesis. Shrink $I_S$ if necessary so that some
complex polydisc neighbourhood of it lies in $D_S$. Successive
application of the one-variable identity theorem, one active coordinate
at a time along that polydisc, gives equality on a nonempty open subset
of $D_S$. The difference of the two sides is holomorphic on the
connected domain $D_S$, so the identity theorem gives equality
throughout $D_S$, in particular at $\bz_S$. The section must pass
through $\by$, since freezing a tail coordinate at a nonreal value
would leave no real box from which to initiate the identity-theorem
argument.

\emph{General parameters.} We now pass to parameters with infinitely
many active coordinates. Three convergences occur: in
$W^{2,\infty}(\Ghat;\C^3)$ for the deformations, in
$L^\infty(\R^3;\C^{3\times3})$ for the coefficients, and in
$\cL(H^{-1/2}(\Ghat),H^{1/2}(\Ghat))$ for the resulting boundary
operators, the last being the topology in which \cref{thm:scalar} is
stated. Let $\bz$ be arbitrary with
$\sum_jb_j\abs{z_j-y_j}<\delta_N$, and truncate,
\[
\bz^{(N)}:=\by+(z_1-y_1,\dots,z_N-y_N,0,0,\dots) ,
\]
so that $\bz^{(N)}-\by$ has finite support and the previous stage
applies. By \eqref{eq:ambient-complex} and \cref{lem:ambient},
\[
\norm{D\mathcal T_{\by,\bz^{(N)}}-D\mathcal T_{\by,\bz}}_{L^\infty}
\le
C_E\sum_{j>N}b_j\abs{z_j-y_j}
\xrightarrow[N\to\infty]{}0 ,
\]
the tail of a convergent series; equivalently
$\rr_{\bz^{(N)}}\to\rr_{\bz}$ in $W^{2,\infty}(\Ghat;\C^3)$, because
$\sum_jb_j\abs{z_j-y_j}<\infty$. Hence
$\mathcal A_{\by,\bz^{(N)}}\to\mathcal A_{\by,\bz}$ in
$L^\infty(\R^3;\C^{3\times3})$ by \cref{lem:volume-elliptic}. The
resolvent identity and the uniform bound on
$\mathsf L_{\by,\cdot}^{-1}$ then give
$\mathsf L_{\by,\bz^{(N)}}^{-1}\to\mathsf L_{\by,\bz}^{-1}$ in
$\cL(\dot H^{-1},\dot H^1)$. Composing with the fixed trace maps,
\[
\mathsf W_{\by,\bz^{(N)}}\longrightarrow\mathsf W_{\by,\bz}
\qquad\text{in }\cL\big(H^{-1/2}(\Ghat),H^{1/2}(\Ghat)\big).
\]
This convergence is in the operator topology of \cref{thm:scalar} and
uses the uniform mapping estimate already established, not any property
of the kernel, so the identity survives the limit in that topology.
On the kernel side $\rr_{\bz^{(N)}}\to\rr_{\bz}$ uniformly, so
$G_{0,\bz^{(N)}}\to G_{0,\bz}$ pointwise off the diagonal, and
\eqref{eq:G0-majorant} holds uniformly in $N$; dominated convergence
passes the pairing to the limit. Dominated convergence is used only for
that kernel pairing at regular densities; the $H^{-1/2}\to H^{1/2}$
operator bound and the operator-norm convergence come from the
trace--Newton realization. Equality therefore persists for
arbitrary, in general infinitely supported, $\bz$.

\emph{General densities.} For $\phi\in C^{0,1}(\Ghat)$ the right-hand
side of \eqref{eq:newton-kernel} is a classically defined integral, and
the identity just proved shows that the associated integral operator
admits the bounded extension $\mathsf W_{\by,\bz}$ of
\eqref{eq:Wnewton}. By uniqueness of continuous extension from the
dense subspace $C^{0,1}(\Ghat)\subset H^{-1/2}(\Ghat)$, that extension
\emph{is} $\Vop^0_{\bz}$, and \eqref{eq:newton-kernel} is to be read in
this sense for general $\phi\in H^{-1/2}(\Ghat)$.
\end{proof}

\begin{remark}[The fixed distributional source]\label{rem:newton-jacobian}
The choice of density in \eqref{eq:phys-density} is dictated by the
Piola-pulled scalar term. That term
of the Piola-pulled EFIE \eqref{eq:fixed-form} contains
$\int_{\Ghat}G_\kappa(\rr_{\by}(\xhat)-\rr_{\by}(\yhat))
\phi(\yhat)\ds_{\yhat}$ with \emph{no} factor $J_{\by}(\yhat)$. The
corresponding physical density is therefore $J_{\by}^{-1}\phi$. The
surface Jacobian carried by the physical distribution
$\gamma_{\by}^*\varrho_{\by}$ cancels on pullback, leaving the
parameter-independent source \eqref{eq:source-cancel}. This is the
trace-potential counterpart of \cref{prop:cancel}.
\end{remark}

\subsubsection{The Helmholtz correction and completion}

Split
\begin{equation}\label{eq:helm-split}
G_{\kappa,\bz}=G_{0,\bz}+N_{\kappa,\bz},
\qquad
N_{\kappa,\bz}(\xhat,\yhat)
:=h_\kappa\big(\sqrt{q_{\bz}(\xhat,\yhat)}\big),
\end{equation}
where $h_\kappa(s):=(e^{\mathrm i\kappa s}-1)/(4\pi s)$ for $s\ne0$ and
$h_\kappa(0):=\mathrm i\kappa/(4\pi)$ is entire. No radiating Helmholtz
inverse is needed: the singular part is governed by the coercive
operator \eqref{eq:Lv}, and the correction is regular at the diagonal.

\begin{lemma}[Regular Helmholtz correction]\label{lem:helm-corr}
There is $\eps_N>0$ such that on every $(\bb,\eps_N)$-admissible tube
the operator
$(\mathsf N_{\kappa,\bz}\phi)(\xhat):=\int_{\Ghat}
N_{\kappa,\bz}(\xhat,\yhat)\phi(\yhat)\ds_{\yhat}$
extends uniformly to
$\mathsf N_{\kappa,\bz}:H^{-1/2}(\Ghat)\to H^{1/2}(\Ghat)$ and depends
holomorphically on $\bz$ there, in every coordinate, on
finite-dimensional sections, and along admissible complex lines.
\end{lemma}

\begin{proof}
By \cref{lem:complex-separation},
$\abs{\sqrt{q_{\bz}(\xhat,\yhat)}}\simeq\abs{\xhat-\yhat}$ uniformly on
the admissible set. Away from the diagonal
\[
D_{\xhat}\sqrt{q_{\bz}}
=
\frac{D_{\Ghat}\rr_{\bz}(\xhat)^\top
\big(\rr_{\bz}(\xhat)-\rr_{\bz}(\yhat)\big)}
{\sqrt{q_{\bz}(\xhat,\yhat)}} ,
\]
so \eqref{eq:Fcomplex} and complex separation give
$\abs{D_{\xhat}\sqrt{q_{\bz}}}+\abs{D_{\yhat}\sqrt{q_{\bz}}}\le C$
a.e., uniformly in $\bz$. Since $h_\kappa$ and $h_\kappa'$ are entire,
hence bounded on the bounded set of values taken by
$\sqrt{q_{\bz}}$,
\[
\abs{N_{\kappa,\bz}}
+\abs{D_{\xhat}N_{\kappa,\bz}}
+\abs{D_{\yhat}N_{\kappa,\bz}}\le C .
\]
The correction is nonsingular at the diagonal but not smooth across it:
$h_\kappa(s)=\mathrm i\kappa/(4\pi)-\kappa^2s/(8\pi)+O(s^2)$, and the
linear term carries $s=\sqrt{q_{\bz}}$, which retains distance-type
regularity in $(\xhat,\yhat)$. The displayed bounds are all the
argument uses.
The Schur test then gives $\mathsf N_{\kappa,\bz}:L^2\to H^1$
uniformly. The transposed kernel satisfies the same estimates, so
duality gives $\mathsf N_{\kappa,\bz}:H^{-1}\to L^2$, and complex
interpolation, using $[H^{-1},L^2]_{1/2}=H^{-1/2}$ and
$[L^2,H^1]_{1/2}=H^{1/2}$, yields $H^{-1/2}\to H^{1/2}$.

For holomorphy, dominated convergence gives weak holomorphy on the
dense class of Lipschitz densities, with \eqref{eq:kernelbound} as
majorant. The uniform operator bound just obtained upgrades this to
operator-norm holomorphy. For $\norm\phi\le1$ and $\norm\lambda\le1$
the scalar functions $\zeta\mapsto\lambda(\mathsf N_{\kappa,\cdot}\phi)$
are holomorphic and uniformly bounded on a fixed disc, so the Cauchy
formula bounds the second difference of their difference quotients
uniformly in $\phi$ and $\lambda$; taking suprema gives a Cauchy family
in $\cL(H^{-1/2},H^{1/2})$, which is complete. Along a complex
line the same applies, since
$\rr_{\bz+\tau\bw}=\rr_{\bz}+\tau\ps_{\bw}$ makes
$q_{\bz+\tau\bw}$ a polynomial in $\tau$, so
$N_{\kappa,\bz+\tau\bw}(\xhat,\yhat)$ is holomorphic in $\tau$ for
fixed $\xhat\ne\yhat$ and the uniform bound just proved supplies the
local boundedness.
\end{proof}

\begin{proposition}[Trace--Newton proof of \cref{thm:scalar}]
\label{prop:newton-proof}
The conclusions of \cref{thm:scalar} follow from
\cref{prop:newton,lem:helm-corr}.
\end{proposition}

\begin{proof}
Let $\brho$ be $(\bb,\eps)$-admissible with
$\eps\le\min\{\eps_N,\delta_N\}$ and let $\bz\in\mathcal O_{\brho}$.
Choose $\by\in\U$ coordinatewise with $\abs{z_j-y_j}<\rho_j-1$; then
$\sum_jb_j\abs{z_j-y_j}\le\sum_jb_j(\rho_j-1)\le\eps$, so
\cref{prop:newton} applies with constants independent of $\by$, $\bz$
and $\brho$. By \eqref{eq:helm-split},
\eqref{eq:newton-kernel} and \cref{lem:helm-corr},
\[
\Vop_{\bz}=\mathsf W_{\by,\bz}+\mathsf N_{\kappa,\bz}
\]
on a neighbourhood of $\bz$, both terms lying uniformly in
$\cL(H^{-1/2}(\Ghat),H^{1/2}(\Ghat))$ and being operator-norm
holomorphic there in every coordinate, on finite-dimensional sections,
and along admissible complex lines, by
\cref{lem:volume-elliptic,prop:newton,lem:helm-corr}. This gives
\textup{(i)} and the uniform bound \textup{(ii)}, with a constant
independent of the admissible radii. Every step above also invokes
\cref{lem:complex-separation}, through the majorant
\eqref{eq:G0-majorant} and through \cref{lem:helm-corr}, so the
threshold must not exceed $\eps_0$; \cref{thm:scalar} therefore holds
with $\eps_1:=\min\{\eps_0,\eps_N,\delta_N\}$. For \textup{(iii)}, the local representations agree on overlaps. If
$\bz$ lies in neighbourhoods based at two real parameters
$\by_1,\by_2\in\U$, then \eqref{eq:newton-kernel} gives
\[
\mathsf W_{\by_1,\bz}\phi
=
\Vop^0_{\bz}\phi
=
\mathsf W_{\by_2,\bz}\phi
\]
first for $\phi\in C^{0,1}(\Ghat)$. The middle expression involves only
the boundary map $\rr_{\bz}$, not the auxiliary base point. The
identity then extends to every $\phi\in H^{-1/2}(\Ghat)$ by
continuity. The
local Newton realizations therefore glue to a single operator-valued
holomorphic map on all of $\mathcal O_{\brho}$. Restricting the patched family to any admissible
complex line yields the directional statement, including its
infinite-coordinate part.
\end{proof}


\subsection{The matrix-weighted vector term}

\begin{lemma}[Vector single-layer form]\label{lem:vector}
There exists $\eps_2>0$ such that on every $(\bb,\eps_2)$-admissible
complex domain the bilinear forms
\begin{equation}\label{eq:Bvec}
B_{\bz}^{\rm vec}(\widehat\jj,\widehat\vv)
:=
\iint_{\Ghat\times\Ghat}
G_{\kappa,\bz}(\xhat,\yhat)
\big[F_{\bz}(\yhat)\widehat\jj(\yhat)\big]\cdot
\big[F_{\bz}(\xhat)\widehat\vv(\xhat)\big]
\ds_{\yhat}\ds_{\xhat}
\end{equation}
are bounded on $\Hpar{-1/2}{\Ghat}\times\Hpar{-1/2}{\Ghat}$ uniformly
in $\bz$, and $\bz\mapsto B_{\bz}^{\rm vec}$ is holomorphic in every
coordinate, jointly holomorphic on finite-dimensional sections, and
holomorphic along admissible complex lines in the sense of
\cref{thm:scalar}\,\textup{(iii)}, with values in the Banach space of
bounded bilinear forms.
\end{lemma}

\begin{proof}
By \eqref{eq:Fcomplex}, multiplication by $F_{\bz}$ is bounded on
$H^s(\Ghat)$ for $\abs{s}\le1$, uniformly in $\bz$, and by
\cref{lem:tangential-identification} it maps
$\Hpar{-1/2}{\Ghat}$ boundedly into $H^{-1/2}(\Ghat;\C^3)$. Applying
\cref{thm:scalar} componentwise to $F_{\bz}\widehat\jj$ and pairing
with $F_{\bz}\widehat\vv\in H^{-1/2}(\Ghat;\C^3)$ gives
\[
\abs{B_{\bz}^{\rm vec}(\widehat\jj,\widehat\vv)}
\le
C_V\norm{F_{\bz}\widehat\jj}_{H^{-1/2}}
\norm{F_{\bz}\widehat\vv}_{H^{-1/2}}
\le
C_VC_F^2
\norm{\widehat\jj}_{\Hpar{-1/2}{\Ghat}}
\norm{\widehat\vv}_{\Hpar{-1/2}{\Ghat}} .
\]
The map $\bz\mapsto F_{\bz}$ is affine, hence entire, with values in
$W^{1,\infty}$. Multiplication and bilinear evaluation are continuous.
Holomorphy therefore follows from \cref{thm:scalar}\,\textup{(iii)} by
composition, valid for $\eps_2:=\min\{\eps_0,\eps_1\}$, since
\cref{thm:scalar} is applied componentwise and \eqref{eq:Fcomplex}
requires $\eps_0$. The directional statement follows in the same way. Along an admissible
line, $F_{\bz+\tau\bw}=F_{\bz}+\tau D_{\Ghat}\ps_{\bw}$ is affine in
$\tau$. The map $\tau\mapsto\Vop_{\bz+\tau\bw}$ is operator-norm
holomorphic by \cref{thm:scalar}\,\textup{(iii)}. Hence the composition
is holomorphic in $\tau$. This directional form is used in \cref{lem:segment}.
\end{proof}

\section{Shape holomorphy of the Maxwell EFIE}\label{sec:operator}

\subsection{Parametric holomorphy}\label{sec:paramhol}

We adopt the notion of parametric holomorphy used in Refs.~\refcite{JerezSchwabZech2017},~\refcite{ChkifaCohenSchwab2015} and~\refcite{DolzHenriquez2024}.

The following is Definition~2.6 of Ref.~\refcite{DolzHenriquez2024}, with the same tube, admissibility budget and polyellipse bound.

\begin{definition}[$(\bb,p,\eps)$-holomorphy]\label{def:hol}
Let $Y$ be a complex Banach space, $\bb\in\ell^p(\N)$ with $0<p<1$, and
$\eps>0$. A continuous map $u:\U\to Y$ is
$(\bb,p,\eps)$-holomorphic if for every $(\bb,\eps)$-admissible $\brho$
it admits an extension
\[
u:\mathcal O_{\brho}\longrightarrow Y
\]
which is holomorphic in each coordinate separately, jointly
holomorphic on every finite-dimensional section, and which satisfies
\[
\sup_{\bz\in\mathcal E_{\brho}}\norm{u(\bz)}_Y\le C_\eps
\]
with $C_\eps$ depending on $\eps$ but not on $\brho$.
\end{definition}

As fixed in \cref{sec:geometry}, $\mathcal O_{\brho}$ is the domain of
holomorphy and $\mathcal E_{\brho}$ carries the uniform bound required
by the approximation theory; by \eqref{eq:EinO} a uniform bound on
$\mathcal O_{\brho}$ implies the latter. All estimates below provide
this stronger bound, and we record them on $\mathcal O_{\brho}$ whenever
that is what the proof gives.

For $\bz\in\mathcal O_{\brho}$ define the complex-bilinear form
\begin{equation}\label{eq:complexEFIE}
\widehat a_{\bz}(\widehat\jj,\widehat\vv)
:=
\mathrm i\kappa\,B_{\bz}^{\rm vec}(\widehat\jj,\widehat\vv)
+
\frac1{\mathrm i\kappa}
\pair{\Vop_{\bz}\divh\widehat\jj}{\divh\widehat\vv}_{H^{1/2},H^{-1/2}} .
\end{equation}
By \cref{rem:bilinear}, every duality in \eqref{eq:complexEFIE} is
complex bilinear.

\begin{theorem}[Main operator theorem]\label{thm:operator}
Under \cref{ass:reference,ass:param} there exists $\eps_A>0$ such that
\begin{equation}\label{eq:operator-map}
\U\ni\by\longmapsto\widehat\Aop_{\by}\in\cL(\Xhat,\Xhat')
\end{equation}
is $(\bb,p,\eps_A)$-holomorphic, with the extension holomorphic along
admissible complex lines in the sense of
\cref{thm:scalar}\,\textup{(iii)}: on every open disc $D\subset\C$ with
$\bz+\tau\bw\in\mathcal O_{\brho}$ for all $\tau\in D$, the map
$\tau\mapsto\widehat\Aop_{\bz+\tau\bw}$ is operator-norm holomorphic.
More precisely, for every $(\bb,\eps_A)$-admissible $\brho$, formula
\eqref{eq:complexEFIE} defines an operator
$\widehat\Aop_{\bz}\in\cL(\Xhat,\Xhat')$ for every
$\bz\in\mathcal O_{\brho}$ with
\begin{equation}\label{eq:CA}
\sup_{\bz\in\mathcal O_{\brho}}
\norm{\widehat\Aop_{\bz}}_{\cL(\Xhat,\Xhat')}\le C_A ,
\end{equation}
$C_A$ independent of $\brho$, and $\widehat\Aop_{\bz}$ coincides with
\eqref{eq:pullback-def} for real $\bz=\by\in\U$.
\end{theorem}

\begin{proof}
Let $\eps_A:=\min\{\eps_1,\eps_2\}$ and let $\brho$ be
$(\bb,\eps_A)$-admissible. By \cref{lem:vector},
\[
\abs{B_{\bz}^{\rm vec}(\widehat\jj,\widehat\vv)}
\le
C\norm{\widehat\jj}_{\Hpar{-1/2}{\Ghat}}
\norm{\widehat\vv}_{\Hpar{-1/2}{\Ghat}},
\]
and by \cref{thm:scalar},
\[
\Big|\pair{\Vop_{\bz}\divh\widehat\jj}{\divh\widehat\vv}\Big|
\le
C_V\norm{\divh\widehat\jj}_{H^{-1/2}}
\norm{\divh\widehat\vv}_{H^{-1/2}} .
\]
Adding and using \eqref{eq:Xnorm} gives
$\abs{\widehat a_{\bz}(\widehat\jj,\widehat\vv)}\le
C\norm{\widehat\jj}_{\Xhat}\norm{\widehat\vv}_{\Xhat}$ uniformly over
admissible complex domains, so $\widehat a_{\bz}$ determines a unique
$\widehat\Aop_{\bz}\in\cL(\Xhat,\Xhat')$ satisfying \eqref{eq:CA}.

Both terms of \eqref{eq:complexEFIE} are holomorphic with values in the
Banach space $\mathcal B_2(\Xhat;\C)$ of bounded complex-bilinear
forms, by \cref{thm:scalar}\,\textup{(iii)} and \cref{lem:vector}; the
canonical isometry $\mathcal B_2(\Xhat;\C)\cong\cL(\Xhat,\Xhat')$ of
\cref{rem:bilinear} preserves holomorphy.
Finally, for real $\bz=\by\in\U$, comparison of
\eqref{eq:complexEFIE} with \eqref{eq:fixed-form} and
\cref{prop:cancel} shows that $\widehat\Aop_{\by}$ is the Piola pullback
\eqref{eq:pullback-def} of the physical EFIE operator.

This proves the extension, holomorphy and uniform boundedness
requirements of \cref{def:hol}. Continuity of the real restriction is
proved independently in \cref{cor:Lip}, whose proof uses only the
extension, the line holomorphy and the bound \eqref{eq:CA}.
\end{proof}

The following consequence of directional holomorphy will be used in
\cref{cor:Lip} and \cref{lem:complexinverse}.

\begin{lemma}[Cauchy estimate along complex segments]\label{lem:segment}
Let $\bz_0,\bz_1$ be such that the segment
$\bz_t:=\bz_0+t\bw$, $\bw:=\bz_1-\bz_0$, $t\in[0,1]$, has the property
that for some $\sigma>0$ and every $t\in[0,1]$ the image of the open
disc $D_\sigma:=\{\tau:\abs\tau<\sigma\}$ under
$\tau\mapsto\bz_t+\tau\bw$ lies in a $(\bb,\eps_A)$-admissible tube.
Then
\[
\norm{\widehat\Aop_{\bz_1}-\widehat\Aop_{\bz_0}}
_{\cL(\Xhat,\Xhat')}
\le
\frac{C_A}{\sigma} .
\]
\end{lemma}

\begin{proof}
By \cref{thm:operator} the map $\tau\mapsto\widehat\Aop_{\bz_t+\tau\bw}$
is operator-norm holomorphic on the open disc $D_\sigma$ and bounded there by $C_A$, by \eqref{eq:CA}.
The Cauchy estimate gives
\[
\norm{\tfrac{d}{dt}\widehat\Aop_{\bz_t}}_{\cL(\Xhat,\Xhat')}
\le
\frac{C_A}{\sigma}
\]
uniformly for $t\in[0,1]$. Integrating along the segment yields the
result.
\end{proof}

\begin{corollary}[Lipschitz perturbation estimate]\label{cor:Lip}
There is $C_L>0$ such that
\begin{equation}\label{eq:Lip}
\norm{\widehat\Aop_{\by}-\widehat\Aop_{\widetilde\by}}
_{\cL(\Xhat,\Xhat')}
\le
C_L\sum_{j\ge1}b_j\abs{y_j-\widetilde y_j}
\qquad\text{for all }\by,\widetilde\by\in\U .
\end{equation}
In particular, since $\bb\in\ell^1(\N)$, the map
\eqref{eq:operator-map} is continuous on $\U$ equipped with the product
topology.
\end{corollary}

\begin{proof}
Let $\by,\widetilde\by\in\U$ with $\by\ne\widetilde\by$ and put
$\bw:=\widetilde\by-\by$. Since $\bb\in\ell^1(\N)$,
\[
S:=\sum_{j\ge1}b_j\abs{w_j}\le2\sum_{j\ge1}b_j<\infty .
\]
Since $\U$ is convex,
$\by+t\bw\in\U$ for every $t\in[0,1]$.

We allocate half of the admissibility budget to the active direction
and reserve half to ensure $\rho_j>1$ for every $j$, as
\cref{def:admissible} requires. Set
\[
\sigma:=\frac{\eps_A}{2S},
\qquad
\rho_j-1:=\sigma\abs{w_j}+\eta_j ,
\]
where the $\eta_j>0$ are chosen with $\sum_jb_j\eta_j\le\eps_A/2$, for
instance $\eta_j=\eps_A\,2^{-j-1}/b_j$. Then
\[
\sum_{j\ge1}b_j(\rho_j-1)
=
\sigma S+\sum_{j\ge1}b_j\eta_j
\le
\frac{\eps_A}{2}+\frac{\eps_A}{2}
=\eps_A ,
\]
so $\brho$ is $(\bb,\eps_A)$-admissible. For $t\in[0,1]$ and $\abs{\tau}<\sigma$, the $j$th coordinate of
$\by+t\bw+\tau\bw$ lies at distance at most $\abs{\tau w_j}$ from
$[-1,1]$, because $\by+t\bw\in\U$. Thus
\[
\abs{\tau w_j}<\sigma\abs{w_j}<\rho_j-1 ,
\]
so the whole disc lies in $\mathcal O_{\brho}$. \Cref{lem:segment} therefore gives
\[
\norm{\widehat\Aop_{\widetilde\by}-\widehat\Aop_{\by}}
\le
\frac{C_A}{\sigma}
=
\frac{2C_A}{\eps_A}\sum_{j\ge1}b_j\abs{y_j-\widetilde y_j} ,
\]
which is \eqref{eq:Lip} with $C_L=2C_A/\eps_A$.

For the continuity statement, given $\delta>0$ choose $J$ with
$\sum_{j>J}b_j<\delta/(4C_L)$; then the product-topology neighbourhood
$\{\widetilde\by:\abs{y_j-\widetilde y_j}<\delta/(2C_L\sum_{j\le
J}b_j),\ j\le J\}$ maps into the $\delta$-ball around
$\widehat\Aop_{\by}$, by \eqref{eq:Lip}.
\end{proof}

\subsection{Incident data}\label{sec:solution}

For a plane wave
\begin{equation}\label{eq:plane}
\EE^{\rm inc}(x)=\pp\,e^{\mathrm i\kappa\ddir\cdot x},
\qquad
\ddir\in\mathbb S^2,\quad\pp\cdot\ddir=0 ,
\end{equation}
the EFIE datum is $f_{\Gamma_{\by}}=-\Iop_{\Gamma_{\by}}\gamma_D
\EE^{\rm inc}$, and the pulled-back datum
$\widehat f_{\by}:=\Pop_{\by}'f_{\Gamma_{\by}}$ satisfies, for
tangential $\widehat\vv$,
\begin{equation}\label{eq:data-pull}
\widehat f_{\by}(\widehat\vv)
=
-\int_{\Gamma_{\by}}\gamma_D\EE^{\rm inc}\cdot\Pop_{\by}\widehat\vv\ds
=
-\int_{\Ghat}\EE^{\rm inc}(\rr_{\by}(\xhat))\cdot
F_{\by}(\xhat)\widehat\vv(\xhat)\ds_{\xhat} ,
\end{equation}
the second equality using \eqref{eq:piola}, the change of variables,
and the tangentiality of $F_{\by}\widehat\vv$ on $\Gamma_{\by}$. Once
again the surface Jacobian cancels, and no normal field or tangential
projector needs to be differentiated. For $\bz\in\mathcal O_{\brho}$ we
define
\begin{equation}\label{eq:data-complex}
\widehat f_{\bz}(\widehat\vv)
:=
-\int_{\Ghat}
e^{\mathrm i\kappa\ddir\cdot\rr_{\bz}(\xhat)}\,
\pp\cdot F_{\bz}(\xhat)\widehat\vv(\xhat)\ds_{\xhat} .
\end{equation}

\begin{lemma}[Holomorphy of plane-wave data]\label{lem:data}
There exists $\eps_f>0$ such that
$\by\mapsto\widehat f_{\by}\in\Xhat'$ is
$(\bb,p,\eps_f)$-holomorphic. The same holds for any incident field that is entire in $x$ and whose
values and first spatial derivatives are uniformly bounded on the
admissible complex point set $\{\rr_{\bz}(\Ghat)\}$, which is
sufficient for the $H^{1/2}$ multiplier argument below.
\end{lemma}

\begin{proof}
On any admissible bounded complex domain,
\[
\bz\longmapsto e^{\mathrm i\kappa\ddir\cdot\rr_{\bz}}
\]
is entire with values in $W^{1,\infty}(\Ghat)$. This follows because
$\bz\mapsto\rr_{\bz}$ is affine with values in $W^{2,\infty}$ and the
exponential is entire. The map is uniformly bounded on the admissible
set.

Since $\bz\mapsto F_{\bz}$ is affine, the product is holomorphic and
uniformly bounded as a multiplier on $H^{1/2}$. The functional
\eqref{eq:data-complex} is therefore bounded on $\Xhat$, because
$\Hpar{-1/2}{\Ghat}\ni\widehat\vv\mapsto
\int_{\Ghat}\phi\cdot\widehat\vv$ is bounded for
$\phi\in H^{1/2}(\Ghat;\C^3)$. The claimed holomorphy follows.
\end{proof}

\subsection{Uniform invertibility from pointwise nonresonance}

Let $\Sigma_E(D_{\by})$ denote the set of interior Maxwell eigenvalues
of $D_{\by}$ with perfect electric conductor boundary condition.

\begin{assumption}[Pointwise nonresonance]\label{ass:invert}
For every $\by\in\U$ one has $\kappa^2\notin\Sigma_E(D_{\by})$.
\end{assumption}

In Refs.~\refcite{DolzHenriquez2024} and~\refcite{JerezSchwabZech2017} a parameter-uniform bound on the inverse is imposed as a hypothesis. Here it is instead
\emph{derived}: pointwise nonresonance, the Lipschitz estimate
\eqref{eq:Lip} and compactness of $\U$ in the product topology suffice,
so no uniform assumption beyond \cref{ass:invert} is made. The same
mechanism upgrades pointwise discrete stability to parametric
uniformity in \cref{lem:disc-uniform}.

\begin{lemma}[Uniform invertibility]\label{lem:uniforminverse}
Under \cref{ass:reference,ass:param,ass:invert} there is
$C_{\rm inv}<\infty$ with
\begin{equation}\label{eq:uniforminverse}
\sup_{\by\in\U}
\norm{\widehat\Aop_{\by}^{-1}}_{\cL(\Xhat',\Xhat)}\le C_{\rm inv} .
\end{equation}
\end{lemma}

\begin{proof}
By \cref{ass:invert} and the classical unique solvability of the EFIE at nonresonant frequencies~\cite{BuffaHiptmair2003,Nedelec}, $\Aop_{\kappa,\Gamma_{\by}}\in\cL_{\rm iso}(X(\Gamma_{\by}),
X(\Gamma_{\by})')$ for every $\by\in\U$, hence by
\eqref{eq:inverse-transfer} $\widehat\Aop_{\by}\in
\cL_{\rm iso}(\Xhat,\Xhat')$ for every $\by\in\U$. By \cref{cor:Lip}
the map $\by\mapsto\widehat\Aop_{\by}$ is continuous from $\U$ with the
product topology into $\cL(\Xhat,\Xhat')$. The set
$\cL_{\rm iso}(\Xhat,\Xhat')$ of isomorphisms is open in
$\cL(\Xhat,\Xhat')$, and $\Aop\mapsto\Aop^{-1}$ is continuous on it
(\cref{lem:neumann}); the composition
\[
\by\longmapsto\widehat\Aop_{\by}
\longmapsto\widehat\Aop_{\by}^{-1}
\longmapsto\norm{\widehat\Aop_{\by}^{-1}}_{\cL(\Xhat',\Xhat)}
\]
is therefore a continuous real-valued function on $\U$. Since
$\U=[-1,1]^\N$ is compact in the product topology by Tychonoff's
theorem, that function attains a finite maximum, which is
\eqref{eq:uniforminverse}.
\end{proof}

\begin{remark}[Where compactness of $\U$ is used]
\Cref{lem:uniforminverse} is the only use of compactness of $\U$ in the
continuous analysis. The same compactness argument reappears at the
discrete level in \cref{lem:disc-uniform}. The argument is qualitative:
it gives no control of $C_{\rm inv}$ in terms of the distance from
$\kappa^2$ to $\bigcup_{\by\in\U}\Sigma_E(D_{\by})$, so the bound
degenerates as a resonance is approached.
\end{remark}

\begin{lemma}[Complex stability of inversion]\label{lem:complexinverse}
Under \cref{ass:reference,ass:param,ass:invert} there exists
$\eps_{\rm inv}>0$ such that for every $(\bb,\eps_{\rm inv})$-admissible
$\brho$ and every $\bz\in\mathcal O_{\brho}$ one has
$\widehat\Aop_{\bz}\in\cL_{\rm iso}(\Xhat,\Xhat')$ with
\[
\sup_{\bz\in\mathcal O_{\brho}}
\norm{\widehat\Aop_{\bz}^{-1}}_{\cL(\Xhat',\Xhat)}\le2C_{\rm inv} ,
\]
and $\bz\mapsto\widehat\Aop_{\bz}^{-1}$ is holomorphic in every
coordinate and jointly holomorphic on finite-dimensional sections.
\end{lemma}

\begin{proof}
Choose first $\eps_{\rm inv}\le\eps_A/4$ and let $\brho$ be
$(\bb,\eps_{\rm inv})$-admissible, $\bz\in\mathcal O_{\brho}$.
Decompose $\bz=\by+\bw$ as in \cref{lem:complex-separation}, so that
$\by\in\U$ and $S:=\sum_j\abs{w_j}b_j\le\eps_{\rm inv}$. If $S=0$ then
$\bw=\mathbf 0$, so $\bz=\by\in\U$ and the assertion is
\cref{lem:uniforminverse}; assume therefore $S>0$. Each factor
$\mathcal O_{\rho_j}$ is convex and contains both $y_j$ and $z_j$, so
every point $\bm\zeta=\by+t\bw$, $0\le t\le1$, satisfies
$\sum_jb_j\abs{\zeta_j-y_j}\le S\le\eps_{\rm inv}\le\eps_A/4$. The
additional admissibility margin prevents the available Cauchy radius
from degenerating near the boundary of the original tube. Set $\sigma:=\eps_A/(2S)$ and define an auxiliary radius sequence
$\widetilde\brho$ by $\widetilde\rho_j-1:=(1+\sigma)\abs{w_j}+\eta_j$ with $\eta_j>0$
chosen so that $\sum_jb_j\eta_j\le\eps_A/4$. Then
\[
\sum_{j\ge1}b_j(\widetilde\rho_j-1)
=
S+\sigma S+\sum_{j\ge1}b_j\eta_j
\le
\frac{\eps_A}{4}+\frac{\eps_A}{2}+\frac{\eps_A}{4}
=\eps_A ,
\]
so $\widetilde\brho$ is $(\bb,\eps_A)$-admissible, and for $t\in[0,1]$,
$\abs\tau<\sigma$ the point $\by+(t+\tau)\bw$ lies at distance at most
$(1+\sigma)\abs{w_j}<\widetilde\rho_j-1$ from $[-1,1]$ in the $j$th
coordinate.
\Cref{lem:segment} then gives
\[
\norm{\widehat\Aop_{\bz}-\widehat\Aop_{\by}}
\le
\frac{C_A}{\sigma}
=
\frac{2C_A}{\eps_A}\,S
\le
C\eps_{\rm inv} ,
\qquad C:=\frac{2C_A}{\eps_A} .
\]
Choosing
\[
\eps_{\rm inv}
:=\min\Big\{\frac{\eps_A}{4},\ \frac1{2CC_{\rm inv}}\Big\}
\]
and invoking \cref{lem:uniforminverse},
\[
\norm{\widehat\Aop_{\by}^{-1}
\big(\widehat\Aop_{\bz}-\widehat\Aop_{\by}\big)}
\le
C_{\rm inv}C\eps_{\rm inv}\le\tfrac12 ,
\]
so \cref{lem:neumann} yields invertibility of $\widehat\Aop_{\bz}$
together with
$\norm{\widehat\Aop_{\bz}^{-1}}\le2C_{\rm inv}$. Holomorphy of
$\bz\mapsto\widehat\Aop_{\bz}^{-1}$ follows from
\cref{thm:operator} and holomorphy of inversion on the open set of
Banach space isomorphisms.
\end{proof}

\begin{theorem}[Parametric holomorphy of the surface current]
\label{thm:current}
Under \cref{ass:reference,ass:param,ass:invert} there exists
$\eps_{\rm sol}>0$ such that the unique solution of the pulled-back EFIE
\begin{equation}\label{eq:pullback-efie}
\widehat\Aop_{\by}\widehat\jj_{\by}=\widehat f_{\by}
\quad\text{in }\Xhat' ,
\qquad\by\in\U,
\end{equation}
defines a $(\bb,p,\eps_{\rm sol})$-holomorphic map
$\U\ni\by\mapsto\widehat\jj_{\by}\in\Xhat$. Moreover
$\widehat\jj_{\by}=\Pop_{\by}^{-1}\jj_{\by}$, where $\jj_{\by}\in
X(\Gamma_{\by})$ is the physical surface current on $\Gamma_{\by}$.
\end{theorem}

\begin{proof}
Take $\eps_{\rm sol}\le\min\{\eps_A,\eps_f,\eps_{\rm inv}\}$. For every $(\bb,\eps_{\rm sol})$-admissible radius sequence, both
$\bz\mapsto\widehat\Aop_{\bz}^{-1}$ and $\bz\mapsto\widehat f_{\bz}$
are holomorphic on the corresponding tube and uniformly bounded there,
by \cref{lem:complexinverse} and \cref{lem:data}. The evaluation map
$\cL(\Xhat',\Xhat)\times\Xhat'\to\Xhat$ is continuous and bilinear,
hence $\widehat\jj_{\bz}=\widehat\Aop_{\bz}^{-1}\widehat f_{\bz}$ is
holomorphic with
$\sup_{\bz}\norm{\widehat\jj_{\bz}}_{\Xhat}\le
2C_{\rm inv}\sup_{\bz}\norm{\widehat f_{\bz}}_{\Xhat'}$. The last
statement follows from \eqref{eq:pullback-def} and
$\widehat f_{\by}=\Pop_{\by}'f_{\Gamma_{\by}}$.
\end{proof}

\section{Dimension-robust approximation of the EFIE operator and
solution}\label{sec:consequences}

\subsection{Mixed parametric derivatives and sparse polynomial
approximation}\label{sec:derivatives}

Let
$\mathcal F:=\{\bnu=(\nu_j)_{j\ge1}\in\N_0^\N:\abs{\supp\bnu}<\infty\}$
and write $\bnu!=\prod_j\nu_j!$, $\abs{\bnu}=\sum_j\nu_j$ and
$\bbeta^{\bnu}=\prod_j\beta_j^{\nu_j}$.

\begin{theorem}[Anisotropic Cauchy estimates]\label{thm:cauchy}
There exists a positive sequence $\bbeta=(\beta_j)_{j\ge1}$ with
$\beta_j\le Cb_j$, and constants $C_A,C_{\jj}$, such that for every
$\bnu\in\mathcal F$
\begin{equation}\label{eq:derivative-op}
\sup_{\by\in\U}
\norm{\partial_{\by}^{\bnu}\widehat\Aop_{\by}}_{\cL(\Xhat,\Xhat')}
\le
C_A\,\abs{\bnu}!\,\bbeta^{\bnu} ,
\end{equation}
and, under \cref{ass:invert},
\begin{equation}\label{eq:derivative-current}
\sup_{\by\in\U}
\norm{\partial_{\by}^{\bnu}\widehat\jj_{\by}}_{\Xhat}
\le
C_{\jj}\,\abs{\bnu}!\,\bbeta^{\bnu} .
\end{equation}
The factorial here is $\abs{\bnu}!$ and not $\bnu!$. The combinatorial
factor produced by optimizing the Cauchy radii under the budget
\eqref{eq:admbudget} cannot be absorbed while retaining $\bnu!$; see
\cref{app:cauchy}.
\end{theorem}

\begin{proof}
Both families satisfy the uniform \emph{tube} bound
\eqref{eq:tubebound} required by \cref{app:cauchy}, which is stronger
than what \cref{def:hol} asks: for $\widehat\Aop_{\bz}$ this is
\eqref{eq:CA}, and for $\widehat\jj_{\bz}$ it follows from
\cref{lem:complexinverse} and \cref{lem:data} as in the proof of
\cref{thm:current}, both bounds being stated on
$\mathcal O_{\brho}$. Applying the multivariate Cauchy estimate of
\cref{app:cauchy} on coordinate discs contained in
$\mathcal O_{\brho}$, and optimizing the radii subject to
\eqref{eq:admbudget}, gives
\eqref{eq:derivative-op}--\eqref{eq:derivative-current}.
\end{proof}

The estimates quantify the anisotropy of the parametric dependence.
They are not used to derive the coefficient summability below; for that
purpose we apply the polyellipse-holomorphy result directly, in
\cref{prop:summability}. The coefficient estimates are instead obtained from the polyellipse theory of Refs.~\refcite{CohenDeVoreSchwab2011},~\refcite{ChkifaCohenSchwab2015},~\refcite{CohenDeVore2015} and~\refcite{ZechSchwab2020}, summarized next.

The Bernstein-polyellipse bounds in \cref{def:hol} naturally lead to
tensor-product Legendre expansions, which we use in two
normalizations. Let $(P_n)_{n\ge0}$ be the univariate
Legendre polynomials normalized by $\norm{P_n}_{L^\infty([-1,1])}=1$,
put $P_{\bnu}(\by):=\prod_jP_{\nu_j}(y_j)$, so that
\begin{equation}\label{eq:Pnorm}
\norm{P_{\bnu}}_{L^\infty(\U)}=1
\qquad\text{for every }\bnu\in\mathcal F ,
\end{equation}
and let $L_{\bnu}:=\prod_j\sqrt{2\nu_j+1}\,P_{\nu_j}$ be the
corresponding family, orthonormal in $L^2(\U,\mu)$ with $\mu$ the
uniform product probability measure. Write $p_{\bnu}$ and $l_{\bnu}$
for the coefficients of $u$ in the two families.

\begin{remark}[Why not a Taylor expansion about the origin]
\label{rem:notaylor}
A Taylor expansion about the origin is not justified under
\cref{def:hol}. The tube $\mathcal O_{\rho_j}$ has width $\rho_j-1$
about $[-1,1]$, and the admissibility budget \eqref{eq:admbudget}
forces $\rho_j-1<1$ as soon as $b_j\ge\eps$. Holomorphy on such a thin
neighbourhood does not make the Taylor series about $0$ converge at
$y_j=\pm1$, since the radius of convergence is governed by the nearest
complex singularity to the origin rather than by continuation along the
real interval. The polydisc formulations of $(\bb,\eps)$-holomorphy used
for Taylor gPC expansions, as in Sec.~3 of Ref.~\refcite{ZechSchwab2020}, require extensions to origin-centred polydiscs
$\prod_j\{\abs{z_j}<\rho_j\}$ with $\rho_j>1$, a genuinely stronger
hypothesis than \cref{def:hol}. The geometry of \cref{sec:geometry}
produces tubes rather than such polydiscs, so we use Legendre
expansions, the system associated with the polyellipse bounds in
\cref{def:hol}.
\end{remark}

\begin{proposition}[Legendre summability and unconditional convergence]
\label{prop:summability}
Let $Y$ be a complex Banach space, $\bb\in\ell^p(\N)$ with $0<p<1$, and
let $u:\U\to Y$ be $(\bb,p,\eps)$-holomorphic in the sense of
\cref{def:hol}. Then the $L^\infty$-normalized coefficient sequence
$(\norm{p_{\bnu}}_Y)$ belongs to the \emph{monotone} class
\[
\ell^p_m(\mathcal F)
:=\Big\{(c_{\bnu}):\
\big(\widehat c_{\bnu}\big)\in\ell^p(\mathcal F),\
\widehat c_{\bnu}:=\sup_{\bmu\ge\bnu}c_{\bmu}\Big\}
\subset\ell^p(\mathcal F),
\]
and $u=\sum_{\bnu}p_{\bnu}P_{\bnu}$ unconditionally in
$L^\infty(\U;Y)$. The same holds for the orthonormal coefficients
$(\norm{l_{\bnu}}_Y)$, and
$u=\sum_{\bnu}l_{\bnu}L_{\bnu}$ unconditionally in $L^\infty(\U;Y)$.
\end{proposition}

This follows from Theorem~2.2 of Ref.~\refcite{ChkifaCohenSchwab2015}, which supplies
both normalizations directly under hypotheses that \cref{def:hol}
satisfies verbatim: their $\mathrm{HA}(p,\eps)$ asks for a uniformly
bounded map admitting, for every $\brho$ with $\rho_j>1$ and
$\sum_j(\rho_j-1)b_j\le\eps$, a coordinatewise holomorphic extension to
a product of open sets containing the Bernstein ellipses, with
$\sup_{\bz\in\mathcal E_{\brho}}\norm{u(\bz)}\le C_\eps$. The tubes in
\eqref{eq:tube} provide such neighbourhoods and contain the Bernstein
ellipses by \eqref{eq:EinO}.

For the orthonormal statement no further import is needed. With
$c_{\bnu}:=\prod_{j}\sqrt{2\nu_j+1}\ge1$ one has
$L_{\bnu}=c_{\bnu}P_{\bnu}$, hence
\begin{equation}\label{eq:legnorm}
l_{\bnu}L_{\bnu}=p_{\bnu}P_{\bnu},
\qquad
l_{\bnu}=c_{\bnu}^{-1}p_{\bnu},
\qquad
\norm{l_{\bnu}}_Y\le\norm{p_{\bnu}}_Y .
\end{equation}
The two expansions have identical summands, so unconditional
$L^\infty(\U;Y)$ convergence of one is that of the other; and since
$\sup_{\bmu\ge\bnu}\norm{l_{\bmu}}_Y
\le\sup_{\bmu\ge\bnu}\norm{p_{\bmu}}_Y$, the monotone envelope of
$(\norm{l_{\bnu}}_Y)$ is dominated by that of $(\norm{p_{\bnu}}_Y)$ and
lies in $\ell^p(\mathcal F)$ with it. The Banach-valued character of the
argument is thus confined to the single imported statement for
$p_{\bnu}$. Related coefficient estimates and polyellipse frameworks appear in Refs.~\refcite{CohenDeVoreSchwab2011},~\refcite{CohenDeVore2015},~\refcite{ZechSchwab2020},~\refcite{JerezSchwabZech2017} and~\refcite{DolzHenriquez2024}.

Two features of \cref{prop:summability} are used below: unconditional
$L^\infty(\U;Y)$ convergence, which identifies the sum of the series
with $u$ on all of $\U$, and monotone-envelope membership, which permits
downward closed index sets.

\Cref{prop:summability} is not a formal consequence of the pointwise
Cauchy bounds of \cref{thm:cauchy}; the obstruction is the factor
$\abs{\bnu}!/\bnu!$ produced by resummation, and is discussed at the end
of \cref{app:cauchy}.

\begin{theorem}[Best $N$-term Legendre approximation]\label{thm:sparse}
Let $Y$ be a complex Banach space and let $u:\U\to Y$ be
$(\bb,p,\eps)$-holomorphic with $\bb\in\ell^p(\N)$, $0<p<1$; by
\cref{thm:operator,thm:current} this covers $u(\by)=\widehat\Aop_{\by}$
with $Y=\cL(\Xhat,\Xhat')$ and $u(\by)=\widehat\jj_{\by}$ with
$Y=\Xhat$. Then there are nested, \emph{downward closed} sets
$\Lambda_N\subset\mathcal F$ with $\abs{\Lambda_N}=N$ such that
\begin{equation}\label{eq:Nterm}
\sup_{\by\in\U}
\Big\|u(\by)-\sum_{\bnu\in\Lambda_N}p_{\bnu}P_{\bnu}(\by)\Big\|_Y
\le C\,N^{-(1/p-1)} ,
\end{equation}
and, if $Y$ is in addition a Hilbert space, nested downward closed sets
$\Lambda_N'$ with $\abs{\Lambda_N'}=N$ and
\begin{equation}\label{eq:Nterm-leg}
\Big\|u-\sum_{\bnu\in\Lambda_N'}l_{\bnu}L_{\bnu}\Big\|_{L^2(\U,\mu;Y)}
\le C\,N^{-(1/p-1/2)} .
\end{equation}
Both rates are independent of the number of activated parametric
dimensions.
\end{theorem}

\begin{proof}
Let $\Lambda_N$ index the $N$ largest entries of the monotone envelope
of $(\norm{p_{\bnu}}_Y)$; such sets may be chosen nested and downward
closed. By \cref{prop:summability} that envelope lies in
$\ell^p(\mathcal F)\subset\ell^1(\mathcal F)$, so Stechkin's inequality
gives $\sum_{\bnu\notin\Lambda_N}\norm{p_{\bnu}}_Y\le CN^{-(1/p-1)}$. The series converges
to $u$ unconditionally in $L^\infty(\U;Y)$, again by
\cref{prop:summability}, so with \eqref{eq:Pnorm}
\[
\sup_{\by\in\U}
\Big\|u(\by)-\sum_{\bnu\in\Lambda_N}p_{\bnu}P_{\bnu}(\by)\Big\|_Y
\le
\sum_{\bnu\notin\Lambda_N}\norm{p_{\bnu}}_Y
\le
CN^{-(1/p-1)} ,
\]
which is \eqref{eq:Nterm}. No structure on $Y$ beyond completeness is
used, which is why the $L^\infty$ normalization \eqref{eq:Pnorm} is the
convenient one here. When $Y$ is Hilbert, Parseval in $L^2(\U,\mu;Y)$
applied to the orthonormal system $(L_{\bnu})$ gives
$\norm{u-\sum_{\Lambda'_N}l_{\bnu}L_{\bnu}}^2_{L^2(\U,\mu;Y)}
=\sum_{\bnu\notin\Lambda'_N}\norm{l_{\bnu}}_Y^2$, and Stechkin for
$\ell^p\subset\ell^2$ yields \eqref{eq:Nterm-leg}. Parseval is the only
step using the Hilbert structure, and it is the reason
\eqref{eq:Nterm-leg} is not asserted for general Banach $Y$.
\end{proof}

\begin{remark}[Scope of the statements]\label{rem:sparse-caveats}
Because \cref{prop:summability} supplies the monotone envelope, the
index sets in \cref{thm:sparse} may be taken downward closed, so the
truncations are amenable to sparse interpolation and Smolyak quadrature
without requiring prior identification of the $N$ largest
coefficients. Chkifa, Cohen and Schwab show in addition, in Theorem~3.1 of Ref.~\refcite{ChkifaCohenSchwab2015}, that the same rate is retained when the Legendre projection is replaced by
interpolation at suitable points; we do not pursue that here.

The two statements of \cref{thm:sparse} use the two normalizations for
different reasons: \eqref{eq:Nterm} rests only on
$\norm{P_{\bnu}}_{L^\infty(\U)}=1$ and holds for every Banach $Y$, in
particular for $Y=\cL(\Xhat,\Xhat')$, whereas \eqref{eq:Nterm-leg}
rests on Parseval for the orthonormal family $(L_{\bnu})$ and is
asserted only for Hilbert $Y$, so it applies to $Y=\Xhat$ but not,
without further argument, to the operator family.
\end{remark}

\begin{corollary}[Sparse surrogates for the operator family]
\label{cor:operator-surrogate}
Let $\bb\in\ell^p(\N)$ with $0<p<1$ and let \cref{ass:param} hold. Then
the map $\by\mapsto\widehat\Aop_{\by}$ admits sparse polynomial
surrogates in $\cL(\Xhat,\Xhat')$ converging at the
dimension-independent best $N$-term rate $N^{-(1/p-1)}$ of
\eqref{eq:Nterm}. Under \cref{ass:invert} the same holds for
$\by\mapsto\widehat\Aop_{\by}^{-1}$ and for
$\by\mapsto\widehat\jj_{\by}$.
\end{corollary}

\begin{proof}
Apply \cref{thm:sparse} with $Y=\cL(\Xhat,\Xhat')$, whose hypotheses
are supplied by \cref{thm:operator}; for the inverse and the current
use \cref{lem:complexinverse,thm:current} in place of
\cref{thm:operator}.
\end{proof}

\Cref{thm:operator} thus verifies the hypotheses of the abstract best
$N$-term theory for the Maxwell EFIE family in $\cL(\Xhat,\Xhat')$, the
natural energy-space operator norm. The approximation results therefore
apply to the entire operator family rather than to a single solution,
and give the approximation-theoretic basis for sparse surrogates of the
Galerkin operator on the fixed-reference space; see \cref{prop:matrix}.

\subsection{Far-field holomorphy}\label{sec:farfield}

For the electric potential \eqref{eq:electric-potential} the far-field
pattern generated by $\jj\in X(\Gamma)$ is
\begin{equation}\label{eq:farphys}
\EE^\infty(\obs)
=
\frac{\mathrm i\kappa}{4\pi}
\big(I-\obs\obs^\top\big)
\int_\Gamma e^{-\mathrm i\kappa\obs\cdot y}\jj(y)\ds_y ,
\qquad
\obs\in\mathbb S^2 .
\end{equation}
Formula \eqref{eq:farphys} is consistent with the conventions of
\cref{sec:maxwell}. Writing $\obs=x/\abs x$ and
expanding
$G_\kappa(x,y)\sim
e^{\mathrm i\kappa\abs x}e^{-\mathrm i\kappa\obs\cdot y}/(4\pi\abs x)$
as $\abs x\to\infty$, the two terms of \eqref{eq:electric-potential}
contribute, respectively,
\[
\frac{\mathrm i\kappa}{4\pi}
\int_\Gamma e^{-\mathrm i\kappa\obs\cdot y}\jj\ds
\qquad\text{and}\qquad
-\frac{\mathrm i\kappa}{4\pi}\,\obs\obs^\top
\int_\Gamma e^{-\mathrm i\kappa\obs\cdot y}\jj\ds ,
\]
the second after the substitution
$\nabla_x\rightsquigarrow\mathrm i\kappa\obs$ and one surface
integration by parts; that integration by parts is the same one that
produces the relative sign in \eqref{eq:efie-form}.

Under the Piola transformation,
$\jj_{\by}(\rr_{\by}(\xhat))J_{\by}(\xhat)
=F_{\by}(\xhat)\widehat\jj_{\by}(\xhat)$, so that
\begin{equation}\label{eq:farpull}
\EE^\infty_{\by}(\obs)
=
\frac{\mathrm i\kappa}{4\pi}
\big(I-\obs\obs^\top\big)
\int_{\Ghat}
e^{-\mathrm i\kappa\obs\cdot\rr_{\by}(\xhat)}
F_{\by}(\xhat)\widehat\jj_{\by}(\xhat)\ds_{\xhat} ,
\end{equation}
and once again the surface Jacobian cancels exactly. For $\widehat\jj_{\by}\in\Xhat$, the integral in \eqref{eq:farpull} is
interpreted by duality rather than pointwise as an $L^1$ integral. Its
$m$th component, $m=1,2,3$, is the pairing of $\widehat\jj_{\by}$ with
the tangential test field
\[
\xhat\longmapsto
e^{-\mathrm i\kappa\obs\cdot\rr_{\by}(\xhat)}
F_{\by}(\xhat)^\top\mathbf e_m .
\]
Under \cref{ass:param} one has $\rr_{\by}\in W^{2,\infty}$ and
$F_{\by}\in W^{1,\infty}$, so this test field belongs uniformly to
$W^{1,\infty}(\Ghat)$, hence to $H^{1/2}(\Ghat)$. This regularity
suffices for the duality pairing and for the $C^m(\mathbb S^2)$
estimates below.

\begin{theorem}[Far-field holomorphy]\label{thm:far}
Under the assumptions of \cref{thm:current} there exists
$\eps_\infty>0$, \emph{independent of $m$}, such that for every
$m\in\N_0$ the map
\[
\U\ni\by\longmapsto\EE^\infty_{\by}
\]
is $(\bb,p,\eps_\infty)$-holomorphic with values in
$C^m(\mathbb S^2;\C^3)$. The associated uniform bound $C_{\infty,m}$
does depend on $m$; only the domain of holomorphy is common to all $m$.
Moreover $\EE^\infty_{\by}$ is real analytic on $\mathbb S^2$ for every
$\by\in\U$, and real analytic in the parameter on every
finite-dimensional section of $\U$.
\end{theorem}

\begin{proof}
Fix $\obs\in\mathbb S^2$. The map
\[
(\bz,\xhat)\longmapsto
e^{-\mathrm i\kappa\obs\cdot\rr_{\bz}(\xhat)}F_{\bz}(\xhat)
\]
is holomorphic in every coordinate of $\bz$, jointly holomorphic on
finite-dimensional sections, and holomorphic along admissible complex
lines in the sense of \cref{thm:scalar}\,\textup{(iii)}. It is
nonsingular in $\xhat$ and uniformly bounded in $W^{1,\infty}(\Ghat)$
on admissible complex sets. Hence it defines a holomorphic family of
bounded operators $\Xhat\to\C^3$, uniformly in $\obs$.

Differentiation with respect to $\obs$ introduces only polynomial
factors involving $\rr_{\bz}(\xhat)$, which remain uniformly bounded on
admissible sets. The resulting estimates are uniform in $\obs$. Hence
the same conclusion holds in $C^m(\mathbb S^2;\C^3)$, with an
admissibility threshold independent of $m$; the bound $C_{\infty,m}$
may depend on $m$. Composition with \cref{thm:current} proves the first
claim.

For real analyticity in $\obs$, the integral in \eqref{eq:farpull}
extends to an entire function of $\obs\in\C^3$. Multiplication by
$I-\obs\obs^\top$ preserves entire dependence, and the restriction to
$\mathbb S^2$ is therefore real analytic.
In detail, the integrand is entire in the observation direction and
uniformly bounded on compact complex neighbourhoods, so the chart
representations extend holomorphically to complex neighbourhoods of the
real chart domains and their restrictions are real analytic. The uniform
$C^m$ bounds obtained above do not by themselves imply this, so the
separate argument is required.
\end{proof}

\begin{corollary}[Linear observables]\label{cor:observables}
Every bounded complex-linear functional of the far field inherits
$(\bb,p,\eps_\infty)$-holomorphy, hence the rates of \cref{thm:sparse}.
This includes each far-field coefficient in a vector spherical harmonic
expansion.
\end{corollary}

\begin{remark}[Radar cross sections]\label{rem:rcs}
Quantities involving conjugation, such as
$\abs{\EE^\infty_{\by}(\obs)}^2$ or the radar cross sections considered in Ref.~\refcite{EscapilJerez2024}, are not holomorphic under the
complexification used here, since they involve complex conjugation.
Their restriction to every finite-dimensional section of $\U$ is
nevertheless real analytic, being obtained from a holomorphic map and
its complex conjugate, and mixed derivative estimates follow from
\cref{thm:cauchy} together with the Leibniz rule. These bounds may be
used to derive approximation results for such quadratic observables,
but such an argument is outside the scope of the present paper.
\end{remark}

\section{Implications for boundary element approximation}
\label{sec:galerkin}

We now transfer the fixed-reference holomorphy results to Galerkin
discretization. Because \cref{prop:cancel} poses every parametric
instance on the same reference surface and in the same space $\Xhat$,
one mesh, divergence-conforming trial space, basis and coefficient
space serve every parameter. The Galerkin operator and matrix therefore
inherit the dimension-robust parametric approximation of
\cref{thm:sparse,cor:operator-surrogate}, with constants independent of
$h$ and $n_h$ in the energy norm induced on the coefficient space.
Here \emph{sparse} refers only to the retained parametric polynomial
coefficients; the boundary element matrices remain dense in the spatial
variables. Let
$\{\widehat{\mathcal T}_h\}$ be a shape-regular family of triangulations
of $\Ghat$ and
\begin{equation}\label{eq:BEspace}
\Xhath\subset\Xhat,
\qquad
\dim\Xhath=:n_h<\infty,
\end{equation}
the associated divergence-conforming boundary element spaces of fixed
degree $k\ge0$, that is, rotated Raviart--Thomas (Rao--Wilton--Glisson for $k=0$) spaces~\cite{BuffaHiptmair2003,SauterSchwab2011,Nedelec}, with basis
$\{\varphi_1,\dots,\varphi_{n_h}\}$. The Galerkin problem is: find
$\widehat\jj_{\by,h}\in\Xhath$ with
\begin{equation}\label{eq:galerkin}
\widehat a_{\by}(\widehat\jj_{\by,h},\widehat\vv_h)
=
\widehat f_{\by}(\widehat\vv_h)
\qquad\forall\,\widehat\vv_h\in\Xhath ,
\end{equation}
with matrix and load vector
\begin{equation}\label{eq:matrix}
\big[\mathbf A_{\by,h}\big]_{mn}
=
\widehat a_{\by}(\varphi_n,\varphi_m),
\qquad
\big[\mathbf f_{\by,h}\big]_m=\widehat f_{\by}(\varphi_m) .
\end{equation}
The mesh, the space, its basis, the number $n_h$ of degrees of freedom
and the coefficient space are all independent of $\by$. Every parametric
instance is therefore represented in the same algebraic coordinates, and
no parameter-dependent reference remeshing is introduced. Any geometric
approximation error belongs to the chosen spatial discretization.
To formulate a mesh-independent matrix bound, we equip the coefficient
space with the norm induced by $\Xhat$; such a bound is false in the
Euclidean norm. Let
\begin{equation}\label{eq:Phih}
\Phi_h:\C^{n_h}\to\Xhath,
\qquad
\mathbf c\longmapsto\sum_{n=1}^{n_h}c_n\varphi_n ,
\end{equation}
and equip $\C^{n_h}$ with $\norm{\mathbf c}_{X,h}
:=\norm{\Phi_h\mathbf c}_{\Xhat}$.

\begin{proposition}[Holomorphy of the Galerkin operator and matrix]
\label{prop:matrix}
Let $\widehat\Aop_{\bz,h}\in\cL(\Xhath,\Xhath')$ be the operator
associated with the restriction of $\widehat a_{\bz}$ to
$\Xhath\times\Xhath$. Then $\by\mapsto\widehat\Aop_{\by,h}$ and
$\by\mapsto\mathbf A_{\by,h}$ are $(\bb,p,\eps_A)$-holomorphic, and
\begin{equation}\label{eq:matrixnorm}
\sup_{\bz}
\norm{\mathbf A_{\bz,h}}
_{\cL((\C^{n_h},\norm{\cdot}_{X,h}),(\C^{n_h},\norm{\cdot}_{X,h})')}
=
\sup_{\bz}\norm{\widehat\Aop_{\bz,h}}_{\cL(\Xhath,\Xhath')}
\le C_A ,
\end{equation}
with $C_A$ from \eqref{eq:CA}, independently of $h$. All matrix norms
in this statement are the operator norms induced by the $\Xhat$-norm on
the coefficient space through $\Phi_h$; no $h$-uniform bound in the
Euclidean matrix norm is asserted. The load-vector map
$\by\mapsto\mathbf f_{\by,h}$ is likewise
$(\bb,p,\eps_f)$-holomorphic. Consequently, the Galerkin operator
admits sparse $L^\infty$-normalized Legendre surrogates with uniform
best $N$-term rate $N^{-(1/p-1)}$, by the Banach-valued estimate
\eqref{eq:Nterm} of \cref{thm:sparse}. The same conclusion holds for
its matrix representation in any fixed basis of $\Xhath$. The
approximation constant depends only on $C_A$, $\bb$, $p$ and $\eps_A$;
it is independent of $h$ and, in particular, of the number $n_h$ of
degrees of freedom.
\end{proposition}

\begin{proof}
By \eqref{eq:Phih}, $\mathbf A_{\bz,h}
=\Phi_h'\,\widehat\Aop_{\bz}\,\Phi_h$, where
$\Phi_h':\Xhat'\to(\C^{n_h},\norm{\cdot}_{X,h})'$ is the dual map of
$\Phi_h$ and the load vector
$\mathbf f_{\by,h}=\Phi_h'\widehat f_{\by}$ is measured in the same
dual coefficient norm. Since $\Phi_h$ is an isometry of
$(\C^{n_h},\norm{\cdot}_{X,h})$ onto $\Xhath\subset\Xhat$, this gives
\eqref{eq:matrixnorm} from \eqref{eq:CA}. Holomorphy is inherited from
\cref{thm:operator} and \cref{lem:data} by restriction, and the rate
from \cref{thm:sparse} with $Y=\cL(\Xhath,\Xhath')$; the $L^2$ estimate
\eqref{eq:Nterm-leg} is not used, because $\cL(\Xhath,\Xhath')$ is not
a Hilbert space in the relevant operator norm
(\cref{rem:sparse-caveats}).
\end{proof}

The EFIE bilinear form is not coercive, and discrete stability rests on
a G\aa rding inequality with a discrete Hodge-type splitting; it is asymptotic in $h$~\cite{BuffaHiptmair2003,BuffaChristiansen2003}. We
record what parametric uniformity follows from the fixed-geometry
theory, which we do not prove here.

\begin{assumption}[Pointwise discrete stability]\label{ass:disc}
For every $\by\in\U$ there are $h_{\by},\gamma_{\by}>0$ such that
\[
\inf_{0\ne\widehat\jj_h\in\Xhath}\
\sup_{0\ne\widehat\vv_h\in\Xhath}\
\frac{\abs{\widehat a_{\by}(\widehat\jj_h,\widehat\vv_h)}}
{\norm{\widehat\jj_h}_{\Xhat}\norm{\widehat\vv_h}_{\Xhat}}
\ \ge\ \gamma_{\by}
\qquad\text{for all }h\le h_{\by}.
\]
\end{assumption}

\begin{lemma}[Parameter-uniform discrete stability]
\label{lem:disc-uniform}
Under \cref{ass:disc} there are $h_\star,\gamma_\star>0$, independent
of $\by$ and $h$, such that the discrete inf-sup constant is at least
$\gamma_\star$ for every $\by\in\U$ and every $h\le h_\star$.
\end{lemma}

\begin{proof}
For fixed $h$ write $\gamma_h(\by)$ for the discrete inf-sup constant
in \cref{ass:disc}. Since $\widehat a_{\by}$ is the bilinear form of
$\widehat\Aop_{\by}$ and $\Xhath\subset\Xhat$, the quotient in
\cref{ass:disc} changes by at most
$\norm{\widehat\Aop_{\by}-\widehat\Aop_{\widetilde\by}}
_{\cL(\Xhat,\Xhat')}$ when $\by$ is replaced by $\widetilde\by$, so
\cref{cor:Lip} gives
\begin{equation}\label{eq:gamma-lip}
\abs{\gamma_h(\by)-\gamma_h(\widetilde\by)}
\le
C_L\sum_{j\ge1}b_j\abs{y_j-\widetilde y_j}
\end{equation}
with $C_L$ independent of $h$ and of the choice of $\Xhath$.

Fix $\by\in\U$. By \eqref{eq:gamma-lip} and the product-topology
continuity established in \cref{cor:Lip}, the set
\[
U_{\by}:=\Big\{\widetilde\by\in\U:\
C_L\sum_{j\ge1}b_j\abs{y_j-\widetilde y_j}<\tfrac12\gamma_{\by}\Big\}
\]
is an open product-topology neighbourhood of $\by$, and
$\gamma_h(\widetilde\by)\ge\tfrac12\gamma_{\by}$ there for every
$h\le h_{\by}$. The family $\{U_{\by}\}_{\by\in\U}$ covers $\U$, which
is compact by Tychonoff, so finitely many
$U_{\by_1},\dots,U_{\by_M}$ suffice. Taking
$h_\star:=\min_ih_{\by_i}$ and
$\gamma_\star:=\tfrac12\min_i\gamma_{\by_i}$ gives the claim.
\end{proof}

Two consequences follow. First, the usual Babu\v ska argument gives
parameter-uniform quasi-optimality for $h\le h_\star$.

Second, the discrete solution map inherits parametric holomorphy, by the
discrete counterpart of \cref{lem:complexinverse}. By
\cref{lem:disc-uniform} the real discrete inverses are bounded by
$\gamma_\star^{-1}$ uniformly in $\by$ and in $h\le h_\star$, and the
discrete operators inherit from \cref{cor:Lip} a Lipschitz bound in the
weighted metric with an $h$-independent constant. The Neumann argument
of \cref{lem:neumann} then applies as in \cref{lem:complexinverse}, with
$\gamma_\star^{-1}$ in place of the continuous inverse bound, producing
an $h$-independent threshold $\eps_{\rm disc}>0$ and a common complex
neighbourhood on which the discrete operators remain invertible with a
uniform bound. Composing with the discrete data map gives
$(\bb,p,\eps_{\rm disc})$-holomorphy of $\by\mapsto\widehat\jj_{\by,h}$,
with constants independent of $h$.

\begin{remark}[Scope]\label{rem:discrete-stability}
\Cref{ass:disc} is a statement about each fixed geometry separately and
is not proved here; \cref{lem:disc-uniform} only upgrades it to a
parametrically uniform statement. No fully discrete spatial-convergence
result is claimed.
\end{remark}

\section{Extensions and open problems}\label{sec:open}

Two extensions appear particularly natural. First, the \(C^{1,1}\)
regularity in \cref{ass:reference} is used in the quantitative ambient
realization behind \cref{thm:scalar}, not in the real-geometry
well-posedness of the EFIE itself. Lowering this assumption would
require a replacement for \cref{lem:ambient}; at the Lipschitz endpoint
one must also use the intrinsic Maxwell trace-space constructions of Refs.~\refcite{BuffaCostabelSheen2002} and~\refcite{BuffaHiptmair2003} rather than the
componentwise identifications used here. We do not pursue this
harmonic-analytic extension.

Second, the magnetic boundary operator remains open. Its kernel behaves
like \(\nabla G_\kappa(x-y)=O(|x-y|^{-2})\) and is understood in the
principal-value sense, so it lies outside the weakly singular theory
used for \cref{thm:scalar}: a fixed-reference Piola representation would
have to preserve the odd-kernel cancellation under complex shape
deformation and still yield uniform boundedness in the Maxwell trace
space. The scalar double-layer analysis of Ref.~\refcite{DolzHenriquez2024} is a useful precedent, but the Maxwell principal-value structure requires a
separate argument. Such a step would be the counterpart of the electric
entry established here for Calder\'on-based dielectric transmission formulations~\cite{JerezSchwab2017,EscapilJerez2024,EscapilJerez2026}. On the
discrete side, a full convergence theory would in addition require
parameter-uniform spatial regularity, geometry and quadrature
consistency, and robust preconditioning.

A separate limitation concerns the wavenumber. All constants obtained
here depend on $\kappa$, and that dependence is not tracked: the
exponential factor $e^{\kappa C\operatorname{diam}\Ghat}$ in the proof
of \cref{lem:kernel} is its visible source, and it propagates into every
later bound through \cref{thm:scalar}. Frequency-explicit estimates, in
the spirit of the acoustic analysis of
Ref.~\refcite{HiptmairSchwabSpence2025}, would require tracking $\kappa$
through the ambient realization and the nonresonance hypothesis, and are
not attempted here.

\section{Conclusions}\label{sec:conclusion}

We have proved operator-norm shape holomorphy of the Maxwell EFIE in its
natural trace space for countably parametric \(\ell^p\)-summable
surface deformations, \(0<p<1\). The argument combines an exact surface
Piola pullback with a uniform
\(H^{-1/2}(\Ghat)\to H^{1/2}(\Ghat)\) theorem for the
complex-deformed scalar single layer. Under pointwise nonresonance, the
surface current and far field inherit this regularity. Their Legendre
coefficients are \(\ell^p\)-summable, giving dimension-independent
best \(N\)-term rates, and the operator-level result transfers to
Galerkin matrices on a single reference mesh with constants independent
of the spatial discretization dimension. Because the statement is at the
operator level, one family of parametric surrogates serves every
incident field and every bounded linear observable, which is the regime
of many-query Maxwell scattering. No truncation in the deformation
amplitude is made, in contrast with first-order shape-perturbation
approaches.

The fractional single-layer estimate extends the existing
\(L^2\)-based boundary-operator holomorphy framework to the Maxwell
energy topology. An analogous treatment of the magnetic principal-value
operator would provide the corresponding ingredient for a
shape-holomorphic Maxwell Calder\'on calculus and parametric dielectric
transmission formulations.

\appendix

\section{The polyellipse lies in the tube}\label{app:ellipse}

Parametrize $\mathcal E_\rho$ by $x=a\cos t$, $y=b\sin t$ with
$a=\tfrac12(\rho+\rho^{-1})$, $b=\tfrac12(\rho-\rho^{-1})$, so that
$a^2-b^2=1$. If $\abs x\le1$ the distance to $[-1,1]$ is
$\abs y\le b$. If $x>1$, write $c:=\cos t\in[1/a,1]$; then
\[
d(c)^2:=(ac-1)^2+b^2(1-c^2),
\qquad
\tfrac12 d^2{}'(c)=(a^2-b^2)c-a=c-a<0 ,
\]
so $d^2$ decreases on $[1/a,1]$ and attains its maximum
$b^2(1-a^{-2})=b^4/a^2$ at $c=1/a$, whence $d\le b^2/a\le b$; the case
$x<-1$ is symmetric. In all cases the distance is at most
$b=(\rho-1)(\rho+1)/(2\rho)<\rho-1$, which is \eqref{eq:EinO}.

\section{Uniform ambient realization}\label{app:ambient}

This appendix proves \cref{lem:ambient}. It is separated from the main
text because it is needed only for the trace--Newton realization of
\cref{sec:newton-route}.

\subsection*{Proof of \cref{lem:ambient}}
Fix $\by\in\U$ and consider the isotopy $t\mapsto\rr_{t\by}(\Ghat)$,
$t\in[0,1]$, which is legitimate because $\U$ is convex, so
$t\by\in\U$ and \cref{ass:param} applies uniformly along it.

We first establish the required uniform tubular estimate.

\emph{Claim.} There is $\delta_{\rm tub}>0$, independent of $\by\in\U$
and $t\in[0,1]$, such that
\begin{equation}\label{eq:tubular}
(\xi,\sigma)\longmapsto\xi+\sigma\,\nn_{t\by}(\xi)
\end{equation}
is bi-Lipschitz from
$\Gamma_{t\by}\times(-\delta_{\rm tub},\delta_{\rm tub})$ onto its
image, with inverse bounded uniformly in $\by$ and $t$; here
$\nn_{t\by}$ is the unit normal of $\Gamma_{t\by}$.

We prove the claim quantitatively, treating near and far pairs
separately, and distinguishing the physical graph radius from the
corresponding separation scale on the reference surface. Write
$\Phi_{t\by}(\xi,\sigma):=\xi+\sigma\nn_{t\by}(\xi)$.

\emph{Near pairs.} By \cref{ass:param} the maps $\rr_{t\by}$ are
uniformly $C^{1,1}$, so there are a \emph{physical graph radius}
$r_{\rm graph}>0$ and constants $M,L$, all independent of $\by$ and
$t$, such that in each ball of radius $r_{\rm graph}$ the surface
$\Gamma_{t\by}$ is the graph of a function $f_{t\by}$ over its tangent
plane with
$\norm{\nabla f_{t\by}}_{L^\infty}\le M$ and
$\operatorname{Lip}(\nabla f_{t\by})\le L$. Consequently the normal
field is uniformly Lipschitz,
$\abs{\nn_{t\by}(\xi)-\nn_{t\by}(\eta)}\le C_{\rm n}\abs{\xi-\eta}$
with $C_{\rm n}=C_{\rm n}(M,L)$. Writing
\[
\Phi_{t\by}(\xi,\sigma)-\Phi_{t\by}(\eta,\tau)
=
(\xi-\eta)
+\sigma\big[\nn_{t\by}(\xi)-\nn_{t\by}(\eta)\big]
+(\sigma-\tau)\nn_{t\by}(\eta) ,
\]
we bound the three terms separately. Abbreviate
\[
A:=\xi-\eta,
\qquad
R:=\sigma\big[\nn_{t\by}(\xi)-\nn_{t\by}(\eta)\big],
\qquad
S:=(\sigma-\tau)\nn_{t\by}(\eta) ,
\]
and split $A$ into its components along and orthogonal to
$\nn_{t\by}(\eta)$, writing $A_{\rm t}$ for the orthogonal part.

The graph representation controls the normal component of $A$. Indeed
\[
A\cdot\nn_{t\by}(\eta)
=\big(1+\abs{\nabla f_{t\by}(v)}^2\big)^{-1/2}
\big[f_{t\by}(u)-f_{t\by}(v)-\nabla f_{t\by}(v)\cdot(u-v)\big],
\]
so the Lipschitz bound on $\nabla f_{t\by}$ gives
\begin{equation}\label{eq:tangent-plane}
\abs{A\cdot\nn_{t\by}(\eta)}\le\tfrac12L\abs{u-v}^2
\le\tfrac12L\abs{\xi-\eta}^2 .
\end{equation}
This is the quantitative form of the transversality between the graph
tangent space and the normal direction. Shrinking the physical graph
radius so that $Lr_{\rm graph}\le1$, \eqref{eq:tangent-plane} yields
$\abs{A_{\rm t}}^2=\abs A^2-(A\cdot\nn_{t\by}(\eta))^2
\ge\tfrac34\abs A^2$ on near pairs.

Since $A_{\rm t}$ is orthogonal to $\nn_{t\by}(\eta)$ and $S$ is
parallel to it,
\[
\abs{A+S}\ \ge\ \abs{A_{\rm t}}\ \ge\ \tfrac{\sqrt3}2\abs{\xi-\eta} ,
\]
\[
\abs{A+S}\ \ge\ \abs{A\cdot\nn_{t\by}(\eta)+(\sigma-\tau)}
\ \ge\ \abs{\sigma-\tau}-\tfrac12L\abs{\xi-\eta}^2 .
\]
Averaging the two bounds with equal weights and using
$\abs{\xi-\eta}\le r_{\rm graph}$ once more,
\[
\abs{A+S}
\ \ge\
\Big(\tfrac{\sqrt3}4-\tfrac14Lr_{\rm graph}\Big)\abs{\xi-\eta}
+\tfrac12\abs{\sigma-\tau}
\ \ge\
\tfrac{\sqrt3}8\abs{\xi-\eta}+\tfrac12\abs{\sigma-\tau} ,
\]
after further reducing $r_{\rm graph}$ so that
$Lr_{\rm graph}\le\tfrac{\sqrt3}2$.

It remains to absorb $R$, which is the perturbation caused by the
variation of the normal. By the Lipschitz bound on $\nn_{t\by}$,
$\abs R\le\abs\sigma C_{\rm n}\abs{\xi-\eta}
\le\delta_{\rm tub}C_{\rm n}\abs{\xi-\eta}$, so choosing
\begin{equation}\label{eq:tub-smallness}
\delta_{\rm tub}\le\frac{\sqrt3}{16\,C_{\rm n}}
\end{equation}
gives
\begin{equation}\label{eq:tub-lower}
\abs{\Phi_{t\by}(\xi,\sigma)-\Phi_{t\by}(\eta,\tau)}
=\abs{A+R+S}
\ \ge\
c_1\abs{\xi-\eta}+c_2\abs{\sigma-\tau} ,
\end{equation}
with $c_1:=\sqrt3/16$ and $c_2:=1/2$.
Both constants depend only on $M$ and $L$, hence are independent of
$\by$ and $t$. This is a \emph{near-pair} estimate: it holds for
$\abs{\xi-\eta}\le r_{\rm graph}$, so that $\xi$ and $\eta$ lie in a
common graph patch. Within such a patch the intrinsic and Euclidean
distances are uniformly equivalent,
\begin{equation}\label{eq:intrinsic-equiv}
\abs{\xi-\eta}\ \le\ d_{\Gamma_{t\by}}(\xi,\eta)
\ \le\ \sqrt{1+M^2}\,\abs{\xi-\eta} ,
\end{equation}
the constant depending only on the uniform graph bound $M$.

Estimate \eqref{eq:tub-lower} gives injectivity and a uniform inverse
Lipschitz estimate, with constant $\max\{c_1,c_2\}^{-1}$ independent of
$\by$ and $t$, for pairs lying in a common graph neighbourhood. The
upper Lipschitz bound for $\Phi_{t\by}$ is immediate from $\abs{\nn_{t\by}}=1$ and the
Lipschitz bound on $\nn_{t\by}$.

To convert this into a condition on the reference surface, note that
\eqref{eq:bilip} gives
\[
\abs{\rr_{t\by}(\xhat)-\rr_{t\by}(\yhat)}
\le C_{\rm lip}\,d_{\Ghat}(\xhat,\yhat) ,
\]
and fix a \emph{reference separation}
\begin{equation}\label{eq:rref}
r_{\rm ref}:=\frac{r_{\rm graph}}{2C_{\rm lip}} .
\end{equation}
If $d_{\Ghat}(\xhat,\yhat)<r_{\rm ref}$ then the images lie within
$r_{\rm graph}/2$ of one another, hence in a common local graph
neighbourhood, so the near-pair estimate applies to them.

\emph{Far pairs.} For $d_{\Ghat}(\xhat,\yhat)\ge r_{\rm ref}$,
compactness of $\Ghat$ gives
\[
m_{\rm ref}:=\inf\{\abs{\xhat-\yhat}:
d_{\Ghat}(\xhat,\yhat)\ge r_{\rm ref}\}>0 .
\]
By \eqref{eq:bilip},
$\abs{\rr_{t\by}(\xhat)-\rr_{t\by}(\yhat)}\ge c_{\rm lip}m_{\rm ref}$
uniformly in $\by$ and $t$. Hence, for such pairs,
\[
\abs{\Phi_{t\by}(\xi,\sigma)-\Phi_{t\by}(\eta,\tau)}
\ge
c_{\rm lip}m_{\rm ref}-\abs\sigma-\abs\tau
\ge
\tfrac12c_{\rm lip}m_{\rm ref}
\qquad\text{provided }
\delta_{\rm tub}<\tfrac14c_{\rm lip}m_{\rm ref} ,
\]
so distinct sheets cannot collide.

To convert this constant lower bound into a multiple of
$d_{\Gamma_{t\by}}(\xi,\eta)+\abs{\sigma-\tau}$, we use that
$\operatorname{diam}_{\rm geo}(\Gamma_{t\by})
\le C_{\rm lip}\operatorname{diam}_{\rm geo}(\Ghat)=:D_*$
uniformly in $\by,t$, again by \eqref{eq:bilip}, so that
$d_{\Gamma_{t\by}}(\xi,\eta)+\abs{\sigma-\tau}
\le D_*+2\delta_{\rm tub}$ and the constant lower bound dominates a
fixed multiple of it.

Combine the two cases and decrease $\delta_{\rm tub}$ if necessary.
There are then $c_{\rm tub},C_{\rm tub}>0$, independent of $\by\in\U$
and $t\in[0,1]$, with
\[
c_{\rm tub}\big(d_{\Gamma_{t\by}}(\xi,\eta)+\abs{\sigma-\tau}\big)
\le
\abs{\Phi_{t\by}(\xi,\sigma)-\Phi_{t\by}(\eta,\tau)}
\le
C_{\rm tub}\big(d_{\Gamma_{t\by}}(\xi,\eta)+\abs{\sigma-\tau}\big) ,
\]
which is the claim; in particular the nearest-point projections onto
the $\Gamma_{t\by}$ are well defined and uniformly Lipschitz on the
corresponding tubes.

Estimate \eqref{eq:tubular} is the quantitative form of the standard
tubular-neighbourhood theorem for $C^{1,1}$ hypersurfaces; see Delfour and Zol\'esio~\cite{DelfourZolesio} for the characterization of
$C^{1,1}$ sets through the oriented distance function and the attendant
Lipschitz nearest-point projection. What that theory does not supply,
and what the argument above establishes, is that $\delta_{\rm tub}$,
$c_{\rm tub}$ and $C_{\rm tub}$ may be chosen \emph{independently of
$\by\in\U$ and $t\in[0,1]$}; that is the only reason the estimate is
carried out explicitly here.

Put $\ps_{\by}:=\sum_{j\ge1}y_j\ps_j$. We define the ambient velocity
by an explicit formula rather than by an abstract extension. Fix a
cutoff $\chi\in C^\infty_c((-\delta_{\rm tub},\delta_{\rm tub}))$ with
$\chi\equiv1$ near $0$, and set
\begin{equation}\label{eq:velocity}
V_{\by,t}\big(\xi+s\,\nn_{t\by}(\xi)\big)
:=
\chi(s)\,\ps_{\by}\big(\rr_{t\by}^{-1}(\xi)\big),
\qquad
\xi\in\Gamma_{t\by},\ \abs s<\delta_{\rm tub},
\end{equation}
extended by zero outside the tube. Fix once and for all a ball
$B\subset\R^3$ containing the closure of the $\delta_{\rm
tub}$-tubular neighbourhood of $\Gamma_{t\by}$ for every $\by\in\U$
and every $t\in[0,1]$; this is possible because $\U$ is bounded and
$\delta_{\rm tub}$ is uniform, so all the surfaces $\Gamma_{t\by}$ lie
in one fixed bounded region. Then $V_{\by,t}$ is supported in $B$ for
all $\by$ and $t$, and consequently $\Phi_{\by,t}=\Id$ on
$\R^3\setminus B$ for all $t$; in particular
$\mathcal T_{\by}=\Id$ there, which is the statement of
\cref{lem:ambient}. The definition is legitimate because
\eqref{eq:tubular} makes $(\xi,s)\mapsto\xi+s\nn_{t\by}(\xi)$ a
bi-Lipschitz parametrization of the tube, and it restricts on
$\Gamma_{t\by}$ to $\ps_{\by}\circ\rr_{t\by}^{-1}$, the velocity of the
isotopy. Each factor in \eqref{eq:velocity} is uniformly Lipschitz in
$x$: $\rr_{t\by}^{-1}$ by \eqref{eq:bilip}, $\nn_{t\by}$ and the
tube coordinates by \eqref{eq:tubular}, and $\ps_{\by}$ because
$\bb\in\ell^1(\N)$. Hence
$\norm{V_{\by,t}}_{W^{1,\infty}(\R^3)}\le C$ uniformly in $(\by,t)$.

The $t$-dependence is recorded in four steps. First,
$t\mapsto\rr_{t\by}$ is affine, hence continuous, with values in
$W^{2,\infty}(\Ghat;\R^3)$.

Second, the pulled-back unit normals $\nn_{t\by}\circ\rr_{t\by}$ depend
continuously on $t$ in $W^{1,\infty}$. Indeed, in a chart the unnormalized normal is the cross product of the
two columns of $D_{\Ghat}\rr_{t\by}$, which is affine in $t$ with values
in $W^{1,\infty}$. Its length is bounded below uniformly, since
\cref{ass:param} makes the tangential differential uniformly
nondegenerate. The normalized quotient is therefore continuous in $t$ in
$W^{1,\infty}$. The normals are \emph{not} affine in $t$: normalization
is nonlinear in the first derivatives of $\rr_{t\by}$.

Third, fix one of the finitely many graph charts of
\cref{ass:reference}. On the common tubular neighbourhood furnished by
\eqref{eq:tubular}, the normal-coordinate map
$(\xi,\sigma)\mapsto\xi+\sigma\nn_{t\by}(\xi)$ and its inverse are
built from $\rr_{t\by}$ and $\nn_{t\by}$ by composition, inversion and
nearest-point projection. Each operation is continuous in the
$W^{1,\infty}$ data by the previous step and by the tubular estimate
\eqref{eq:tubular}, which bounds the inverse uniformly. The resulting
maps are therefore continuous, and in particular measurable, in $t$ for
each fixed point.

Fourth, $V_{\by,t}(x)$ is given by the explicit formula
\eqref{eq:velocity}, a composition of $\chi$, $\ps_{\by}$,
$\rr_{t\by}^{-1}$ and the normal coordinates of $x$. By the three
previous steps every factor is measurable in $t$ at fixed $x$.

It remains to treat a point that enters or leaves the moving tube. Fix
$x\in\R^3$. Since $t\mapsto\rr_{t\by}$ is continuous with values in
$W^{2,\infty}$, and in particular uniformly in $C^0$, the function
\[
t\longmapsto\operatorname{dist}\big(x,\Gamma_{t\by}\big)
\]
is continuous. The cutoff $\chi$ is supported strictly inside
$(-\delta_{\rm tub},\delta_{\rm tub})$, so $V_{\by,t}(x)$ vanishes on
the open set of times where
$\operatorname{dist}(x,\Gamma_{t\by})>\operatorname{supp}\chi$ and is
given by the composition above on the open set where it is smaller.
The two open sets cover $[0,1]$ up to the closed set where equality
holds, on which $V_{\by,t}(x)=0$ by continuity of $\chi$. Hence
$t\mapsto V_{\by,t}(x)$ is measurable for every $x$, including points
that cross the boundary of the tube. That is all the flow argument
requires; continuity of $t\mapsto V_{\by,t}$ in $C^{0,1}(\R^3;\R^3)$ is
neither claimed nor needed.

Together with the uniform Lipschitz bound in $x$ obtained above,
Carath\'eodory's existence and uniqueness theorem applies; see Chap.~2 of Ref.~\refcite{CoddingtonLevinson}. The flow
\[
\partial_t\Phi_{\by,t}(x)=V_{\by,t}(\Phi_{\by,t}(x)),
\qquad
\Phi_{\by,0}=\Id ,
\]
is then globally defined, and Gr\"onwall's inequality applied to the
uniform Lipschitz constant $L_V$ of $V_{\by,t}$ gives
\begin{equation}\label{eq:gronwall}
e^{-L_Vt}\abs{x-x'}
\le
\abs{\Phi_{\by,t}(x)-\Phi_{\by,t}(x')}
\le
e^{L_Vt}\abs{x-x'} ,
\end{equation}
the lower bound being the same estimate for the inverse flow.

It remains to identify the flow on $\Ghat$ with the prescribed isotopy.
Let $\xhat\in\Ghat$ and put $\xi:=\rr_{t\by}(\xhat)$, so that $s=0$ and
$\rr_{t\by}^{-1}(\xi)=\xhat$ in \eqref{eq:velocity}. Since
$\chi(0)=1$,
\begin{equation}\label{eq:velocity-onsurface}
V_{\by,t}\big(\rr_{t\by}(\xhat)\big)
=
\ps_{\by}(\xhat)
=
\partial_t\rr_{t\by}(\xhat) ,
\end{equation}
the last equality because $\rr_{t\by}=\Id+t\ps_{\by}$. Hence
$t\mapsto\Phi_{\by,t}(\xhat)$ and $t\mapsto\rr_{t\by}(\xhat)$ solve the
same ordinary differential equation with the same initial value. By the
uniqueness part of Carath\'eodory's theorem they coincide, so
$\Phi_{\by,t}(\xhat)=\rr_{t\by}(\xhat)$ for all $t$. Taking
$\mathcal T_{\by}:=\Phi_{\by,1}$ gives the realization, with
$C_{\mathcal T}=e^{L_V}$.

For the determinant we use the Jacobian formula for the flow of a
Lipschitz field. We avoid differentiating the flow in $x$, which for a
merely Lipschitz velocity would require care, and argue from the
bi-Lipschitz bounds already obtained.

By \eqref{eq:gronwall} the map $\Phi_{\by,1}$ is bi-Lipschitz with
constants $e^{\pm L_V}$. At any point of differentiability,
$\abs{D\Phi_{\by,1}}\le e^{L_V}$ and
$\abs{D\Phi_{\by,1}^{-1}}\le e^{L_V}$, so
\begin{equation}\label{eq:detbound}
\abs{\det D\mathcal T_{\by}}\ \ge\ e^{-3L_V}
\qquad\text{a.e.}
\end{equation}
For the sign, note that $t\mapsto\Phi_{\by,t}$ is a continuous isotopy
from $\Phi_{\by,0}=\Id$, each $\Phi_{\by,t}$ is a homeomorphism of
$\R^3$ equal to the identity outside the fixed ball $B$, and the
Brouwer degree is a homotopy invariant. Hence
$\deg(\Phi_{\by,t},B,\cdot)=\deg(\Id,B,\cdot)=1$ for every $t$, so
$\mathcal T_{\by}=\Phi_{\by,1}$ is orientation preserving and
$\det D\mathcal T_{\by}\ge e^{-3L_V}=:c_{\mathcal T}>0$ a.e.
The extension operator $\mathsf E$ is constructed on the fixed
reference surface. Since $\Ghat$ is a fixed compact $C^{1,1}$
hypersurface it has a fixed tubular neighbourhood, and a constant
extension along the reference normals, followed by a cutoff, gives the
required operator. Explicitly, with
$x=\xhat+s\nn(\xhat)$ in a small fixed tube set
$(\mathsf Eh)(x):=\chi(s)h(\xhat)$ for a cutoff $\chi$, and extend by
zero. The $C^{1,1}$ regularity makes the nearest-point projection
Lipschitz, so
$\norm{\mathsf Eh}_{W^{1,\infty}(\R^3)}\le C_E\norm h_{W^{1,\infty}(\Ghat)}$.

All constants in this construction are uniform in $\by\in\U$ and
$t\in[0,1]$: the tubular radius $\delta_{\rm tub}$ of
\eqref{eq:tubular}, the Lipschitz constants of the nearest-point
projections, the velocity bounds, the flow estimates, and
$C_{\mathcal T},c_{\mathcal T},C_E$. Consequently the threshold $\delta_N$ of
\cref{prop:newton} is independent of the base point.
\qed

\section{A quantitative inverse perturbation lemma}\label{app:inverse}

\begin{lemma}\label{lem:neumann}
Let $X,Y$ be Banach spaces and let $A_0\in\cL_{\rm iso}(X,Y)$. Suppose
that
\[
\norm{A-A_0}_{\cL(X,Y)}<\norm{A_0^{-1}}_{\cL(Y,X)}^{-1} .
\]
Then $A\in\cL_{\rm iso}(X,Y)$ and
\[
\norm{A^{-1}}
\le
\frac{\norm{A_0^{-1}}}{1-\norm{A_0^{-1}}\norm{A-A_0}} .
\]
\end{lemma}

\begin{proof}
Factor $A=A_0[I+A_0^{-1}(A-A_0)]$ and apply the Neumann series.
\end{proof}

\Cref{lem:neumann} is used in \cref{lem:complexinverse} to extend the
uniform real inverse bound to a common complex neighbourhood, and
analogously in \cref{lem:disc-uniform} at the discrete level.

\section{Cauchy estimates under an anisotropic budget}
\label{app:cauchy}

Let $u:\U\to Y$ admit, for every $(\bb,\eps)$-admissible $\brho$, a
holomorphic extension to the tube $\mathcal O_{\brho}$ satisfying
\begin{equation}\label{eq:tubebound}
\sup_{\bz\in\mathcal O_{\brho}}\norm{u(\bz)}_Y\le M
\end{equation}
with $M$ independent of $\brho$. This assumption is stronger than \cref{def:hol}, which requires a
uniform bound only on $\mathcal E_{\brho}$. The families considered in
\cref{thm:cauchy} satisfy the stronger bound \eqref{eq:tubebound}.

The case $\bnu=\mathbf 0$ follows directly from \eqref{eq:tubebound}, so
assume $\abs{\bnu}\ge1$. \Cref{def:admissible} requires $\rho_j>1$ for
every $j$, not only on $\supp\bnu$. We therefore reserve part of the
admissibility budget for the inactive coordinates: the
inactive coordinates receive
\begin{equation}\label{eq:inactive}
\rho_j-1:=\frac{\eps\,2^{-j-1}}{b_j},
\qquad j\notin\supp\bnu,
\qquad\text{so that}\quad
\sum_{j\notin\supp\bnu}(\rho_j-1)b_j\le\frac{\eps}{2} ,
\end{equation}
which is legitimate because $b_j>0$ for all $j$ by the convention
following \eqref{eq:bj}, and the remaining budget $\eps/2$ is spent on
$\supp\bnu$. The multivariate Cauchy formula on the polydisc
$\prod_{j\in\supp\bnu}D(y_j,\rho_j-1)$, contained in
$\mathcal O_{\brho}$, gives
\[
\norm{\partial_{\by}^{\bnu}u(\by)}_Y
\le
M\,\bnu!\prod_{j\in\supp\bnu}d_j^{-\nu_j} ,
\qquad\by\in\U,
\]
with $d_j=\rho_j-1$ the radius available in the $j$th coordinate.
Subject to the remaining budget
$\sum_{j\in\supp\bnu}(\rho_j-1)b_j\le\eps/2$, the choice
\[
\rho_j-1=\frac{\eps\,\nu_j}{2\,b_j\abs{\bnu}},
\qquad j\in\supp\bnu,
\]
is admissible and yields
\[
\norm{\partial_{\by}^{\bnu}u(\by)}_Y
\le
M\,\bnu!
\Big(\frac{2\abs{\bnu}}{\eps}\Big)^{\abs{\bnu}}
\prod_{j\in\supp\bnu}\Big(\frac{b_j}{\nu_j}\Big)^{\nu_j} .
\]
Now $\bnu!\le\prod_j\nu_j^{\nu_j}$, so the product
$\bnu!\prod_j\nu_j^{-\nu_j}$ is at most $1$, while
$\abs{\bnu}^{\abs{\bnu}}\le e^{\abs{\bnu}}\abs{\bnu}!$ by Stirling.
Hence
\[
\norm{\partial_{\by}^{\bnu}u(\by)}_Y
\le
M\,\abs{\bnu}!\prod_{j\in\supp\bnu}
\Big(\frac{2e\,b_j}{\eps}\Big)^{\nu_j} ,
\]
which is \eqref{eq:derivative-op}--\eqref{eq:derivative-current} with
$\beta_j=2e\,b_j/\eps$; the factor $2$ results from reserving half
the budget for the inactive coordinates in \eqref{eq:inactive}. The optimization produces the multinomial combinatorial growth
associated with the total order $\abs{\bnu}$, which
$\abs{\bnu}^{\abs{\bnu}}\le e^{\abs{\bnu}}\abs{\bnu}!$ converts into the
total-order factorial $\abs{\bnu}!$; it does not yield a bound
$C\,\bnu!\,\bbeta^{\bnu}$ with $\beta_j\lesssim b_j$ uniformly over
arbitrary finite supports, which is why $\abs{\bnu}!$ and not $\bnu!$
appears in \eqref{eq:derivative-op}--\eqref{eq:derivative-current}.
Consequently \eqref{eq:derivative-op} is a statement about derivatives
only, obtained from discs inside the tube; the coefficient summability
of \cref{prop:summability} rests instead on the uniform bound over the
Bernstein polyellipses and is not obtained from it by resummation, cf.\
\cref{rem:notaylor}.

\section*{Acknowledgment}

C.~Jerez-Hanckes acknowledges support from the Agencia Nacional de Investigaci\'on y Desarrollo (ANID), Chile, through Fondecyt Regular grant 1231112.

\section*{Statement of Usage of Artificial Intelligence}
Generative artificial intelligence tools were used to assist with language editing and the presentation of the manuscript.

\end{document}